\documentclass[11pt]{amsart}
\numberwithin{equation}{section} %numerazione per sezioni
\usepackage{amssymb, amsmath, amsthm, latexsym}
\usepackage{amsrefs}

\usepackage{bookmark}

\usepackage{graphicx}
\usepackage{color}
\usepackage[a4paper, centering]{geometry}
\DeclareMathOperator{\Tr}{Tr}
\newcommand*{\Trace}[3][\pm]{\Tr^{#1}(#2, #3)}
\newcommand*{\Trp}[2]{\Trace[i]{#1}{#2}}
\newcommand*{\Trm}[2]{\Trace[e]{#1}{#2}}
\newcommand*{\Trl}[3][\lambda]{\Tr_{#1}(#2, #3)}
\newcommand*{\chiut}[1][]{\chi^{#1}_{\{u > t\}}}
\newcommand*{\jump}[1]{\Theta_{#1}}

\newcommand{\spacelambda}[1][\Omega]{\mathcal{B}(#1;[0,1])}

\def\nuint{\widetilde{\nu}}

\newcommand{\A}{\boldsymbol{A}}
\newcommand{\B}{\boldsymbol B}
\newcommand{\bsmall}{\boldsymbol b}
\newcommand{\ucl}[1][n]{u_{#1}^{(\lambda)}}
\newcommand{\F}{F}
\newcommand{\f}{f}
\def\vv{{\bf v}}

\def\DM{{\mathcal{DM}^{\infty}}}
\newcommand*{\BVL}[1][\R^N]{BV(#1)\cap L^{\infty}{(#1)}}

\def\BV{BV(\Omega)}

\newcommand*{\Conv}[1][u]{\mathcal{A}(#1)}

\newcommand{\mean}[1]{\,-\hskip-1.08em\int_{#1}} %media integrale displayed
\newcommand{\Leb}[1]{\mathcal{L}^{#1}} % Misura di Lebesgue

\def\R{\mathbb{R}}

\def\N{\mathbb{N}}

\def\u{{u}}

\def\uh{\widehat{u}}

\def\rel #1{{\overline #1}}
\DeclareMathOperator{\diver}{div}
\DeclareMathOperator{\Div}{div}
\DeclareMathOperator{\dist}{dist}
\DeclareMathOperator{\spt}{supp}

\newcommand{\defeq} {:=}
\newcommand{\medint}{-\kern  -,375cm\int}
\newcommand{\medintinrigo}{-\kern  -,315cm\int}

\newcommand{\eps}{\varepsilon}
 \newcommand{\hh}{{\mathcal H}^{N-1}}

\newcommand{\LLN}{{\mathcal L}^N}
\newcommand{\LLU}{{\mathcal L}^1}

\newcommand{\M}[1]{\mathcal{#1}}    %al posto di \mathcal argomento trra graffe!!!%
\renewcommand{\H}{\M{H}}
\newcommand{\Haus}[1][N-1]{{\mathcal H}^{#1}} % Misura di Hausdorff

\newcommand{\res}{\mathop{\hbox{\vrule height 7pt width .5pt depth 0pt
\vrule height .5pt width 6pt depth 0pt}}\nolimits} % macro per la restrizione
\def\pscal#1#2{\left\langle #1\,,\, #2 \right\rangle}

\def\ut{\widetilde{u}}
\newcommand{\polar}[1][\lambda]{\theta_{#1}}

\newcommand{\Prec}[2][\lambda]{#2^{#1}}

\newcommand{\pair}[2][\lambda]{\left(#2\right)_{{#1}}}
\newcommand{\Pair}[2]{\left[#2\right]_{{#1}}}
\newcommand{\Lpair}[1]{\left(#1\right)_{L}}

\newcommand{\Vpair}[1]{\left(#1\right)_{V}}
\newcommand{\spair}[1]{\left[#1\right]_{{*}}}
\def\upiu{u^+}
\def\umeno{u^-}
\def\uint{{u^{i}}}
\def\uext{{u^{e}}}

\def\vint{{v^{i}}}
\def\vext{{v^{e}}}

\def\Vu{u^V}

\def\laM{\mathcal{L}_M}
\def\laV{\mathcal{L}_V}

\DeclareFontFamily{U}{mathx}{}
\DeclareFontShape{U}{mathx}{m}{n}{<-> mathx10}{}
\DeclareSymbolFont{mathx}{U}{mathx}{m}{n}
\DeclareMathAccent{\widehat}{0}{mathx}{"70}
\DeclareMathAccent{\widecheck}{0}{mathx}{"71}

\def\radon{\mathcal{M}(\Omega)}
\def\radonp{\mathcal{M}_{+}(\Omega)}
\def\radonh{\mathcal{M}_{\mathcal{H}}(\Omega)}

\long\def\taglio#1{}

\newtheorem{definition}{Definition}[section]

\newtheorem{lemma}[definition]{Lemma}

\newtheorem{theorem}[definition]{Theorem}

\newtheorem{proposition}[definition]{Proposition}

\newtheorem{corollary}[definition]{Corollary}
\theoremstyle{remark}
\newtheorem{remark}[definition]{Remark}
\newtheorem{example}[definition]{Example}

\makeatletter
\def\@settitle{\begin{center}%
		\baselineskip14\p@\relax
		\bfseries
		\uppercasenonmath\@title
		\@title
		\ifx\@subtitle\@empty\else
		\\[5ex]%\@subtitle
		\normalsize\mdseries\@subtitle
		\fi
	\end{center}%
}
\def\subtitle#1{\gdef\@subtitle{#1}}
\def\@subtitle{}
\makeatother

\begin{document}
\title[Divergence of the composition of irregular fields with BV functions]
{On the divergence of the composition \\ of irregular fields with BV functions}

\author[G.~Crasta]{Graziano Crasta}
\address{Dipartimento di Matematica ``G.\ Castelnuovo'', 
Sapienza Universit\`a di Roma\\
	P.le A.\ Moro 5 -- I-00185 Roma (Italy)}
\email{graziano.crasta@uniroma1.it}
\author[V.~De Cicco]{Virginia De Cicco}
\address{Dipartimento di Scienze di Base  e Applicate per l'Ingegneria,
 Sapienza Universit\`a di Roma\\
	Via A.\ Scarpa 10 -- I-00185 Roma (Italy)}
\email{virginia.decicco@uniroma1.it}
\author[A.~Malusa]{Annalisa Malusa}
\address{Dipartimento di Matematica ``G.\ Castelnuovo'', Sapienza Universit\`a di Roma\\
	P.le A.\ Moro 5 -- I-00185 Roma (Italy)}
\email{annalisa.malusa@uniroma1.it}

\keywords{Divergence--measure vector fields, functions of bounded variation, coarea formula, Gauss--Green formula, semicontinuity}
\subjclass[2010]{26B30,49Q15,49J45}

\date{April 15, 2026}

\begin{abstract}
We introduce a family of (nonlinear) pairing measures that ensure the validity of the divergence rule for composite functions $\boldsymbol{B}(x,u(x))$, where $\boldsymbol{B}(\cdot,t)$ is a bounded divergence-measure vector field, and  $u$ is a scalar function of bounded variation.
The elements of the family depend on the choice of the pointwise representative of $u$ on its jump set. 
Beyond the standard properties, such as the Coarea and Gauss-Green formulas on sets of finite perimeter, this flexibility allows us to characterize the pairings that ensure the lower semicontinuity of the corresponding functionals
along sequences converging in $L^1$ with controlled precise values. We show that these lower semicontinuous pairings arise as the relaxation of integral functionals defined in Sobolev spaces.
\end{abstract}

\maketitle

\section{Introduction}

This paper establishes a chain rule formula for the distributional divergence of the composition
$\vv(x)=\B(x,u(x))$ defined in an open subset $\Omega\in\R^N$, extending the classical one 
\begin{equation}\label{classical}
	\Div (\B(x,u))=(\Div_x \B)(x,u)+\partial_t\B(x,u)\cdot \nabla u\,,\ \ x\in\Omega
\end{equation}
that holds in the pointwise sense if $\B(x,t)\in C^1(\Omega \times \R;\R^N)$ and $u\in C^1(\Omega)$.

In the present setting, the scalar function $u$ is assumed to be bounded and of bounded variation in an open set $\Omega \subseteq \R^N$, and 
 $\B \colon\Omega\times\R\to \R^N$ is a family of bounded divergence measure vector fields.  
Formulas of the type \eqref{classical} are unifying,
including the following two special cases interesting in themselves: the chain rule formulas, if $\B(x,t)=\B(t)$ is a Lipschitz function, and 
the Leibniz formulas, if $\B(x,t)=G(t)\A(x)$ for $G\in C^1(\R)$ and $\A\in C^1(\Omega;\R^N)$.

\smallskip

\noindent
\emph{The purpose of our study.}
The main challenge in this framework is identifying a natural substitute for the scalar product 
$\partial_t\B(x,u)\cdot \nabla u$ featured in \eqref{classical}, usually called \textsl{pairing} between $\partial_t\B(x,u)$ and $Du$, enabling the extension of the chain rule identity in the sense of measures. 

We introduce a novel, broader family of pairing measures that depend on the choice of the pointwise representative on the jump set. These pairings not only ensure the validity of the chain rule and preserve all the classical properties of known pairings (such as the coarea and Gauss--Green formulas), but they also provide a framework to study variational problems related to the semicontinuity and relaxation of integral functionals in the $BV$ setting.

The choice of this substitute is also closely
connected with the validity of a Gauss--Green formula of the type
\begin{equation}\label{GAGR}
\int_\Omega (\Div_x \B)(x,u)\,dx+\int_\Omega \partial_t\B(x,u)\cdot \nabla u\,dx
= -\int_{\partial\Omega}\B(x,u)\cdot \nu_{\partial\Omega}\,d\H^{N-1}\,,
\end{equation}
valid for $\B$, $u$ as before, and $\Omega$ Lipschitz domain, with inward normal $\nu_{\partial\Omega}$.

Much effort has been made in recent years in order to extend and generalize \eqref{classical},
and so \eqref{GAGR}, in irregular frameworks
(see \cites{Anz,ADM,ChenFrid,CD2,CD3,CDCS,CCDM,CD4}) where a suitable notion of weak normal traces 
of a vector field on countably $\mathcal{H}^{N-1}$-rectifiable sets is available (see \cites{Anz2,AmbCriMan,ComiPayne}).

The strategy to extend \eqref{classical}  is to reinterpret this formula as an equality in the sense of measures, instead of in a pointwise sense, and to introduce a  {\it{pairing measure}} $(\bsmall(x,u),Du)$
between the field $\bsmall(x,u(x))$ (where $\bsmall(x,t):=\partial_t\B (x,t))$) and the measure derivative $Du$ of a  $BV$ bounded function $u$, 
defined on smooth functions as 
$
(\bsmall(x,u),Du):=\Div (\B(x,u))-(\Div_x \B)(x,u).
$
The assumptions on $\bsmall$ ensuring that $(\bsmall(x,u),Du)$ is well-defined as a Radon measure in $\Omega$ are stated in Section \ref{s:ipo}, and include the requirement that $\bsmall(\cdot,t)$, $t\in\R$, is a family of bounded vector fields whose distributional divergence is a finite Radon measure (briefly $\bsmall(\cdot,t)\in \DM(\Omega)$ for every $t\in\R$), and such that the least upper bound of the total variations $|\Div_x \bsmall(\cdot,t)| $, denoted by $\sigma$, is also a Radon measure. Clearly, the same property is shared by the family $\Div_x\B(\cdot,t)$, since $\B(\cdot,t) =\int_0^t\bsmall(\cdot,s)ds$ implies that $\Div_x\B(\cdot,t)=\int_0^t\Div_x\bsmall(\cdot,s)ds$.
Under these weak assumptions on the vector field, but requiring that $u$ is still a regular function, the natural definition in the sense of distributions for $(\bsmall(x,u),Du)$ is
\begin{equation*}
	\langle(\bsmall(x,u),Du),\varphi\rangle:=
-\int_\Omega\B(x,u)\cdot\nabla\varphi\,dx-\int_\Omega F(x,u)\,\varphi\, d\sigma,
\qquad \varphi\in C^\infty_c(\Omega),
\end{equation*}
where $F(x,t)$ is the Radon-Nikodym derivative of the measure $(\Div_x \B)(x,t)$ with respect to $\sigma$.
In \cite{CD4} it is proved that this distribution is, in fact, a finite Radon measure in $\Omega$. 
Nevertheless, if $u$ is a $BV$ function, the second term of the right hand side 
is meaningless, until a pointwise value of $u$, defined $\sigma$-a.e., is given.

\smallskip

\noindent
\emph{The literature on the pairing.}
In the linear case $\B(x,t)=t\A(x)$ it is well established (see \cites{Anz, ChenFrid})  that the extension to $BV$ functions can be done whenever $\A\in \DM(\Omega)$ by choosing on $J_u$ the precise representative $u^*$ defined $\H^{N-1}$-a.e.\ and hence $|\Div\A|$-a.e.\ (see Proposition \ref{p:basicVF}), i.e.
$
(\A,Du):=\Div(u\A)-{u^*}\Div \A.
$
This is usually called {\it{standard (linear) pairing}}. All the main properties and features 
as coarea, Leibniz and Gauss--Green formulas for this pairing are proved in \cite{CD3}.

A nonlinear pairing was proposed in \cite{CD4}, where, under the assumptions on $\bsmall$ stated in Section \ref{s:ipo}, it is shown that the divergence rule extends to $BV$ functions
and $\DM(\Omega)$ vector fields, setting, for every $\varphi\in C^1_c(\Omega)$,
\begin{equation}\label{generalexternal1}
	\langle(\bsmall(x,u),Du),\varphi\rangle:=-\int_\Omega\B(x,u)\cdot\nabla\varphi\,dx
-\frac12\int_\Omega [F(x,u^+)+F(x,u^-)]\,\varphi\, d\sigma.
\end{equation}
Hence, some very general Gauss--Green formulas are proved, which extend to the nonlinear pairing the analogous result proved in \cite{CD3}, by admitting domains of finite perimeter. A fundamental preliminary result is to prove the representation of the weak normal traces of the composite function $v(x)=\B(x, u(x))$ (see Proposition \ref{p:traces} below).

The standard pairings are very useful technical tools in many situations, 
for instance for optimality conditions in plasticity,  which is the original motivation behind their introduction (see, e.g., \cites{Anz3,AnzGia,KoTe}), for the study of 
semicontinuity properties of integral functionals
(see \cites{BouDM,dcl,DCFV1,DCFV2}), for the study of transport equation and of 
conservation laws with discontinuous flux of the form
$u_t + \Div \B(x, u) = 0$
(see \cites{ACDD,AmbCriMan,ADM,CD2,CDD,ChFr1,ChTo2,ChTo,ChToZi,ChenFrid}), 
in the setting of weak formulations of PDEs such as Euler--Lagrange equations associated 
with integral functionals defined in $BV$
\cites{ABCM,MaRoSe,Mazon2016}, 
Dirichlet problems for equations involving the $1$--Laplace operator 
\cites{AVCM,Cas}, 
the Prescribed Mean Curvature problem and capillarity
\cites{LeoSar,LeoSar2}, and
continuum mechanics
\cites{ChCoTo,DGMM,Silh,Schu}.

In contrast, the standard pairing 
proves inadequate when dealing with semicontinuity properties. 
This becomes apparent in \cite{CDM}, where
we characterize the Borel selections $\lambda:\Omega\to [0,1]$  making the functionals
corresponding to \textit{(linear) $\lambda$--pairings}
 \(
\pair{\A,Du}=\Div(u\A)-\Prec{u}\Div \A,
\)
semicontinuous with respect to the strict convergence in $BV$. Here $u^\lambda:=(1-\lambda)u^-+\lambda u^+$ is a Borel
choice of the representative of $u$ on the jump set, between the approximate lower and upper limits $u^-$ and $u^+$.
We show that, in general,
the standard pairing, corresponding to $\lambda=1/2$, does not share this property.
In \cite{CCDM} (see also \cite{CD3} and \cite{CDM}) 
a detailed description of the density of the pairing measure $\pair{\A, Du}$ with respect to $|Du|$ is obtained.
 
Furthermore, in \cite{CD5} the authors prove that, while the original definition is given in a  distributional sense, the pairing measure introduced in \cite{CD4} can be regarded in a variational sense as a relaxed functional with respect to the weak$^*$ convergence in $BV$,  under the crucial assumption that the measure $\sigma$ has an $L^N$ density with respect to the Lebesgue measure. More precisely, 
for every  $\varphi \in C^\infty_c(\Omega)$,
the functional
\begin{equation}
\label{t:generalenuovo1}
\rel{F_\varphi}(u)=\int_\Omega\varphi\ d(\bsmall(x,u),Du) \,, \qquad u\in BV(\Omega)\cap L^\infty(\Omega)
\end{equation}
is the relaxation of
\begin{equation*}
F_\varphi(u):=
\begin{cases}
\displaystyle\int_\Omega \varphi\  \bsmall(x,u)\cdot\nabla u\,dx
& \text{if $u\in W^{1,1}(\Omega)$}\cap L^\infty(\Omega),
\\
+\infty
,
& \text{if}\ u\in (BV(\Omega)\
\setminus W^{1,1}(\Omega))\cap L^\infty(\Omega)\,,
\end{cases}
\end{equation*}
with respect to weak$^*$ convergence in $BV$.

\smallskip

\noindent
\emph{Explanation of our results.}
In the present paper we extend these results to general bounded fields with measure divergence. To this aim, we introduce the family of  {\it{nonlinear $\lambda$-pairings}} which includes all those studied in the literature so far.  We show that  a further very general Gauss--Green formula, extending all the previous ones, holds, so that these pairings can be used for the weak formulation of PDEs. In addition, the nonlinear $\lambda$-pairings allow us to treat the lower semicontinuity and relaxation for integral functionals defined in $BV(\Omega)$, in the spirit of \cites{CDM,CD5}.
 
 The most obvious way to generalize \eqref{generalexternal1} (the standard nonlinear pairing $(\bsmall(x,u),Du)_*$) is to associate with every
Borel function $\lambda\colon\Omega\to [0,1]$, 
an {\it {external nonlinear $\lambda$--pairing}} 
	\begin{equation}
		\label{f:pairlambdaext1}
		\Pair{\lambda}{\bsmall(\cdot,u), Du}= -\left((1-\lambda)F(x, u^-)+\lambda F(x, u^+)\right) \, \sigma + \Div(\B(x,u))\,.
	\end{equation}
Nevertheless, we  show that the external nonlinear pairing is not the right notion for the semicontinuity issue.
 In order to identify a good selection that guarantees lower semicontinuity properties, we need to enlarge the class of these pairings. 
Indeed, we define the
{\it{nonlinear $\lambda$--pairings}}
as the measure 
\begin{equation}
	\label{f:pairlambda1}
	\pair{\bsmall(\cdot,u), Du}= -F(x, \Prec{u})\, \sigma + \Div(\B(x,u))\,.
\end{equation} 
We prove that for every Borel 
selection $\Lambda\colon \Omega \to [0,1]$ there exists a %$\sigma$-measurable 
Borel selection $\lambda\colon \Omega \to [0,1]$ such that 
\(
\Pair{\Lambda}{\bsmall(\cdot,u), Du}=\pair{\bsmall(\cdot,u), Du}
\)
(see Proposition~\ref{legame}).

On the other hand, the nonlinear pairings differ from the external ones only on the jump set of $u$, and hence they share many of the same properties. In many cases these pairings coincide.
If, for instance, $t\mapsto F(x,t)$ is a nondecreasing function on $\R$, 
then there is a canonical identification between all external nonlinear pairings  and nonlinear pairings
 (see Proposition \ref{legame} and Remark \ref{r:leg}).
In particular, this is the case if $\bsmall(x,t)$ can be split as $\bsmall(x,t)=g(t)\A(x)$, $g\in C_b(\R)$, $g\geq 0$, and $\A\in\DM$
(see Example \ref{ex:separated}). The linear case $\bsmall(x,t)=\A(x)$ is recovered for $g(t)=1$, and
$
(\A, Du)_\lambda=[\A, Du]_\lambda
$ in $\Omega$, with the same $\lambda$.
Both pairings  admit similar structure theorems. 
We prove that the pairings of both families \eqref{f:pairlambdaext1} and \eqref{f:pairlambda1} are absolutely continuous with respect to the measure $|Du|$ (see Theorem \ref{chainb4}). Their corresponding densities %$\theta^\lambda$ and $\theta_\lambda$ 
with respect to $|Du|$ can be explicitly written in $J_u$ in terms of the weak normal traces of $\bsmall$ on $J_u$
(see Proposition \ref{p:tracesonJ}). 
On the other hand, in $\Omega\setminus J_u$ all the pairing measures coincide and the density of their diffuse parts 
(which is independent of $\lambda$) can be described by using the weak normal traces of $\bsmall$ on the reduced boundary of the level set $\{u>\widetilde u(x)\}$ (Theorem~\ref{t:diff}). 
This representation of the diffuse part is based on the use of the coarea formula (see Theorem \ref{t:coarea}) valid for the pairings \eqref{f:pairlambdaext1}. 
Finally, in Theorem \ref{t:GG}, Gauss--Green formulas is obtained, for both the families  \eqref{f:pairlambdaext1} and \eqref{f:pairlambda1}, by generalizing together Theorem 6.1 in \cite{CD4} and Theorem 6.3 in \cite{CDM}. 

 In the last part of the paper, we address lower semicontinuity and relaxation results. 
As in \cite{CDM}, we exhibit pairings $(\bsmall(x,u), Du)_L$ lower semicontinuous with respect to a suitable notion of convergence
(see Definition~\ref{d:conv}), 
weaker than strict convergence in $BV$.
Finally, a relaxation formula \eqref{t:generalenuovo1} for integral functionals defined in $W^{1,1}(\Omega)$ can be established by using approximation sequences considered in \cite[Theorem 3.1]{ComiLeo} (see Theorem \ref{t:limsupG}).

\smallskip

\noindent
\emph{Outline of the paper.}
The paper is organized as follows. Section \ref{s:prelim} and Section \ref{s:ipo} contain the necessary notation, the standing assumptions on the  vector fields, and some preliminary results. In Section \ref{s:pair}, we define the families of nonlinear pairings and establish their representations outside the Cantor part of $Du$. Building on this, Sections \ref{s:coarea} and \ref{s:intrepr} present a generalized Coarea formula and the representation of the Cantor part of the pairings, respectively. These tools are applied in Section 7 to derive generalized Gauss-Green formulas. Finally, the variational applications are discussed in the last two sections: Section \ref{s:sc} addresses the lower semicontinuity of the pairings, and Section \ref{s:relax} proves the related relaxation formula.

\section{Notation and preliminary results}
\label{s:prelim}

As a rule, we follow the notation established in \cite{AFP},  recalling in this section only those notions most relevant to our purposes.

In the sequel, \(\Omega\) denotes a nonempty open subset of 
\(\R^N\), $N\geq 2$. 
We write $C^k_c(\Omega, \R^m)$ for the space of
$C^k$ functions  $\varphi\colon\Omega\to\R^m$ with compact support in $\Omega$,
and $C_b(\Omega,\R^m)$ for the space of continuous and bounded functions $\varphi\colon\Omega\to\R^m$.
For $m=1$ we write $C^k_c(\Omega)$ and $C_b(\Omega)$, respectively.
We denote by $\LLN$  the Lebesgue measure in $\R^N$
and by $\hh$ the $(N-1)$--dimensional 
Hausdorff measure.

Given a \(\LLN\)-measurable set \(E\subset\R^N\),
we denote by $E_t$ the set of points of density $t\in[0,1]$, namely
\[
E^t := \left\{x\in\R^N:\
\lim_{\rho\to 0^+} \frac{\LLN(E\cap B_\rho(x))}{\LLN(B_\rho(x))} = t\right\}\,.
\]
The sets \(E^0\), \(E^1\), \(\partial^e E := \R^N\setminus (E^0 \cup E^1)\) are called, 
respectively, the \textsl{measure-theoretic exterior}, 
the \textsl{measure-theoretic interior} and
the \textsl{essential boundary} of \(E\).

\smallskip
Let $u\colon \Omega\to\R$ be a Borel function.
We denote by $\umeno$ and $\upiu$ 
the \textsl{approximate lower limit} and the 
\textsl{approximate upper limit} of $u$,
defined respectively by
\begin{gather*}
u^+(x) := \inf\{t\in\R:\ \{u>t\}\ \text{has density $0$ at $x$}\},
\\
u^-(x) := \sup\{t\in\R:\ \{u>t\}\ \text{has density $1$ at $x$}\}.
\end{gather*}
The function $u$ is \textsl{approximately continuous} at $x\in\Omega$
if $\upiu(x) = \umeno(x)$; in this case, we denote their common value
by $\ut(x)$.

For \(u\in L^1_{{\rm loc}}(\Omega)\),
a point $x\in\Omega$ is a \textsl{Lebesgue point} 
of $u$ (with respect to $\LLN$)
if there exists \(z\in\R\) such that
\[
\lim_{r\rightarrow0^{+}}\frac{1}{\LLN\left(  B_r(x)\right)}\int_{B_r\left(  
x\right)
}\left|  u(y)  -z  \right|  \,dy=0.
\]
In this case, $x$ is a point of approximate
continuity, and $z = \ut(x)$
(see \cite[Proposition~1.163]{FonLeoBook}). 
We denote by \(S_u\subset\Omega\) the set of points where
approximate continuity fails.

We say that \(x\in\Omega\) is an {\sl approximate jump point} of \(u\) if
there exist \(a,b\in\R\) and a unit vector \(\nu\in\R^N\) such that \(a\neq b\)
and
\begin{equation}\label{f:disc}
\begin{gathered}
\lim_{r \to 0^+} \frac{1}{\LLN(B_r^i(x,\nu))}
\int_{B_r^i(x,\nu)} |u(y) - a|\, dy =
\lim_{r \to 0^+} \frac{1}{\LLN(B_r^e(x,\nu))}
\int_{B_r^e(x,\nu)} |u(y) - b|\, dy = 0,
\end{gathered}
\end{equation}
where 
\begin{equation}
\label{f:halfballs}
\begin{split}
B_r^i(x,\nu)  & := \{y\in B_r(x):\ (y-x)\cdot \nu_u(x) > 0\}, 
\\
B_r^e(x,\nu)  & := \{y\in B_r(x):\ (y-x)\cdot \nu_u(x) < 0\}.
\end{split}
\end{equation}
The triplet \((a,b,\nu)\), uniquely determined by \eqref{f:disc} 
up to a permutation
of \((a,b)\) and a simultaneous change of sign of \(\nu\),
is denoted by \((\uint(x), \uext(x), \nu_u(x))\).
We denote by \(J_u\) the set of approximate jump points of \(u\).

\subsection{Measures}

We denote by $\radon$ (resp. $\radonp$) the spaces of all Radon measures (resp.\ nonnegative Radon measures) on $\Omega$. For $\mu\in\radon$ and $\varphi\in C_c(\Omega)$ we use the notation
\[
\langle \mu,\varphi \rangle :=\int_\Omega \varphi\, d \mu.
\]

Given $\mu\in\radon$  we denote by $\mu\res E$ the \text{restriction} 
of $\mu$  to a $\mu$-measurable set $E$, and by 
$|\mu|$, $\mu^+$, $\mu^-$  its {\sl total variation}, and its {\sl positive} and {\sl negative parts}, respectively.

It is well known that $|\mu|, \mu^+,\mu^-$ belong to $\radonp$, and by Hahn's decomposition theorem, we have that  there exist two mutually disjoint Borel sets $\Omega^+_\mu$, $\Omega^-_\mu$ such that
$\Omega=\Omega^+_\mu\cup \Omega^-_\mu$, 
$\mu^+= \mu \res \Omega^+_\mu$, 
$\mu^-= -\mu \res \Omega^-_\mu$.

Given a nonnegative Borel %Radon 
measure $\nu\in\radonp$, we say that $\mu\in\radon$ is 
{\sl absolutely continuous} with respect to $\nu$ (and we write
$\mu \ll \nu$), if $|\mu|(B)=0$ for every set $B$ such that $\nu(B)=0$. We frequently deal with measures not charging  $\Haus$--null sets.
\begin{definition}\label{d:mh}
A Borel measure $\mu$ belongs to $\radonh$ if $|\mu|(B)=0$ for every Borel set $B\subseteq\Omega$ such that $\Haus(B)=0$.
\end{definition}

We say that two  measures $\nu_1$, $\nu_2\in\radonp$ are {\sl mutually 
singular}
(and we write $\nu_1 \perp \nu_2$) 
if there exists a Borel set $E$ such that $|\nu_1|(E)=0$ and 
$|\nu_2|(\Omega\setminus E) = 0$.
By the Radon--Nikod\'ym theorem, given $\nu\in\radonp$, every 
$\mu\in\radon$ can be uniquely decomposed as $\mu=\mu_1+\mu_2$ with
$\mu_1 \ll \nu$ and $\mu_2 \perp \nu$, and there exists a unique Borel function
$\psi_\nu\in L^1(\Omega; \nu)$,
called the density of $\mu$ with respect to $\nu$,
such that $\mu_1=\psi_\nu \nu$.
In particular, since $\mu \ll |\mu|$, then there exists
$\psi\in L^1(\Omega; |\mu|)$, with $|\psi|=1$ $|\mu|$--a.e. in $\Omega$, and such 
that  $\mu = \psi |\mu|$. This is usually called the {\sl polar decomposition}
of $\mu$. 

\smallskip
Given $\mu\in\radon$, we denote by $\mu = \mu^a + \mu^s$ its
Lebesgue decomposition in the absolutely continuous part
$\mu^a \ll \LLN$ and the singular part $\mu^s\perp\LLN$. 

\subsection{Functions of bounded variation}
\label{ss:BV}

We say that \(u\in L^1(\Omega)\) is a \textsl{function of bounded variation} in 
\(\Omega\)
if its distributional gradient \(Du\) is a finite Radon vector valued measure in \(\Omega\).
The vector space of all functions of bounded variation in \(\Omega\)
will be denoted by \(BV(\Omega)\).

If \(u\in BV(\Omega)\), then \(Du\) can be decomposed as
the sum of the absolutely continuous and the singular part with respect
to the Lebesgue measure, i.e.\
$Du = D^a u + D^s u$, and
$D^a u = \nabla u \, \LLN$,
where \(\nabla u\) is the \textsl{approximate gradient} of \(u\),
defined \(\LLN\)-a.e.\ in \(\Omega\)
(see \cite[Section~3.9]{AFP}).
The jump set $J_u$ is countably $\H^{N-1}$--rectifiable,
and  $\H^{N-1}(S_u  \setminus J_u) = 0$
(see \cite[Definition~2.57 and Theorem~3.78]{AFP}).
Moreover, the singular part \(D^s u\) can be further decomposed
as the sum of its Cantor and jump part, i.e.
$D^s u = D^c u + D^j u$,
with $D^c u := D^s u \res (\Omega\setminus S_u)$,
and
\[
D^j u := D^s u \res J_u = (\Prec[i]{u} - \Prec[e]{u}) \, \nu_u\, \Haus\res J_u.
\]
We denote by \(D^d u := D^a u + D^c u\) the diffuse part of the measure 
\(Du\).

At every point $x\in J_u$ we have
$-\infty < \umeno(x) < \upiu(x) < +\infty$ and
\[
\umeno(x) =\min\{\uint(x),\uext(x)\}, \qquad 
\upiu(x) =\max\{\uint(x),\uext(x)\},
\qquad x\in J_u.
\]
In few occasions, to simplify the notation, 
we will choose the orientation on $J_u$ so that $\uint = \upiu$ on $J_u$
(see \cite[\S 4.1.4, Theorem~2]{GMS1}).
In this case, $\nu_u$ coincides $\Haus$-a.e.\ on $J_u$ with the density of the polar decomposition
of $Du$, that, with some abuse of notation, we denote with $\nu_u$, i.e., $Du = \nu_u |Du|$.
In what follows, we extend the functions $\umeno,\upiu$ to
\(\Omega\setminus(S_u\setminus J_u)\) by setting
\[
\umeno=\upiu=\widetilde{u}\quad \text{in}\ \Omega\setminus S_u.
\]

\begin{definition}
Given a Borel function $\lambda \colon \Omega \to [0,1]$ (briefly $\lambda\in \spacelambda$), 
the \textsl{ $\lambda$--representative} of $u\in BV(\Omega)$ is defined by
\begin{equation}\label{f:pr}
\Prec{u}(x):=
\begin{cases}
\widetilde{u}(x), & x\in \Omega \setminus S_u, \\
(1-\lambda(x))\umeno(x)+\lambda(x)\upiu(x), & x\in J_u.
\end{cases}
\end{equation} 
\end{definition}

\begin{remark}
When $\lambda(x) = 1/2$ for every $x\in\Omega$,
the $\lambda$--representative coincides with the {\sl precise representative} $u^* := (\upiu + \umeno)/2$ of $u$.
Since $\Haus(S_u\setminus J_u)=0$ for $u\in BV(\Omega)$, 
the function $\Prec{u}$ is well defined $\Haus$--a.e.\ in $\Omega$.
\end{remark}

The following lemma relates the pointwise behavior  of  \(\chi_{\{u>t\}}\) to that of \(u\),  see \cite[Lemma 2.2]{DCFV2} for the proof.
\begin{lemma} \label{funzcaracter}
	Let $u\in BV(\Omega)$.
	Then, for $\mathcal{L}^1$-a.e. $t\in\R$ the function $\chi_{\{u>t\}}$ belongs to $BV(\Omega)$.
	Moreover, there
	exists a Borel set $N_t\subset\Omega$, with $\H^{N-1}(N_t)=0$, such that 
	the following relation holds:
	\[
	\chi_{\{u^\pm>t\}}(x) =\chi_{\{u>t\}}^{\pm}(x),
	\qquad
	\forall x\in\Omega\setminus N_t.
	\]
\end{lemma}

\smallskip
A \(\LLN\)-measurable set \(E\) compactly contained in $\Omega$ is said to be a \textsl{set of finite perimeter} in $\Omega$ if
its characteristic function $\chi_E$ belongs to $\BV$.
The \textsl{reduced boundary} \(\partial^* E\) of a set of finite perimeter \(E\) is the set of all points
\(x\in \Omega\) in the support of \(|D\chi_E|\) such that the limit
\[
\nuint_E(x) := \lim_{\rho\to 0^+} 
\frac{D\chi_E(B_\rho(x))}{|D\chi_E|(B_\rho(x))}
\]
exists in \(\R^N\) and satisfies \(|\nuint_E(x)| = 1\).
The function \(\nuint_E\colon\partial^* E\to S^{N-1}\) is called
the \textsl{measure-theoretic interior unit normal} to \(E\).

A fundamental result of De Giorgi (see \cite[Theorem~3.59]{AFP}) states that
\(\partial^* E\) is countably \((N-1)\)-rectifiable
and \(|D\chi_E| = \hh\res \partial^* E\).
If \(E\) has finite perimeter in \(\Omega\), Federer's structure theorem
states that
\(\partial^* E\cap\Omega \subset E^{1/2} \subset \partial^e E\)
and \(\H^{N-1}(\Omega\setminus(E^0\cup \partial^e E \cup E^1)) = 0\)
(see \cite[Theorem~3.61]{AFP}).

\subsection{Divergence--measure fields }
\label{ss:div}

We denote by \(\DM(\Omega)\) the space of all
vector fields 
\(\A\in L^\infty(\Omega, \R^N)\)
whose distributional divergence, defined by
\[
\int_\Omega \varphi\, d\Div\A = -\int_{\Omega} \A\cdot\nabla\varphi\, dx \qquad
\forall \varphi\in C^1_c(\Omega),
\]
 is a finite Radon measure in \(\Omega\).

The basic properties of such vector fields are collected in the following proposition.

\begin{proposition}\label{p:basicVF}
Let \(\A\in\DM(\Omega)\), and define
\begin{equation}\label{f:jump0}
\jump{ \Div\A}
:= \left\{
x\in\Omega:\
\limsup_{r \to 0+}
\frac{|\Div \A| (B_r(x))}{r^{N-1}} > 0
\right\}.
\end{equation}  
Then:
\begin{itemize}
\item[(i)] \(\Div\A \in\radonh\);
\item[(ii)] $\jump{\Div\A}$ is a Borel set, \(\sigma\)-finite with respect to 
\(\hh\);
\item[(iii)]\(
\Div\A = \Div^a\A + \Div^c\A + \Div^j\A,
\)
where \(\Div^a\A\) is absolutely continuous with respect to \(\LLN\),
\(\Div^c\A (B) = 0\) for every set \(B\) with \(\hh(B) < +\infty\),
and there exists $f\in L^1(\jump{\Div\A}; \Haus\res \jump{\Div\A})$
such that
\(
\Div^j\A = f\, \hh\res\jump{\Div\A}
\).
\end{itemize}
\end{proposition}
\begin{proof}
The main property (i) is proved in \cite[Proposition 3.1]{ChenFrid}.
The decomposition then follows from \cite[Proposition 2.3] {ADM}.
\end{proof}

The measures $\Div^j\A$ and $ \Div^c\A $ are called, respectively,
the \textsl{jump part} and the \textsl{Cantor part} 
of the measure $\Div\A$, while $\jump{ \Div\A}$ is referred to
as the \textsl{jump set} of $\Div\A$.

\medskip
It is proved in \cite{AmbCriMan} that every $\A\in\DM(\Omega)$ admits weak normal traces on a countably \(\H^{N-1}\)--rectifiable set \(\Sigma\subset\Omega\).
To fix notation, we briefly recall how the weak normal traces are defined.

Let \((\Sigma_i)_{i\in\N}\) be a covering of \(\Sigma\), 
and let $N_i\subseteq \Sigma_i$ be
Borel sets satisfying the following properties:
\begin{itemize}
	\item[(R1)] Each \(\Sigma_i\) is an \textsl{oriented} \(C^1\) hypersurface, equipped 
	with its classical normal vector field \(\nu_{\Sigma_i}\).

	\item[(R2)] The sets \(N_i\subseteq \Sigma_i\) are pairwise disjoint, and \(\hh(\Sigma\setminus \bigcup_i N_i) = 0\).

	\item[(R3)] For every $i\in\N$, there exist two open, bounded sets \(\Omega_i, \Omega'_i\) with \(C^1\) boundary,
	whose exterior normal vectors we denote by \(\nu_{\Omega_i}\) and \(\nu_{\Omega_i'}\), such that
	\(N_i\subseteq \partial\Omega_i \cap \partial\Omega'_i\),
	and
	\[
	\nu_{\Sigma_i}(x) = \nu_{\Omega_i}(x) = -\nu_{\Omega'_i}(x)
	\qquad \forall x\in N_i.
	\]
\end{itemize} 

For every open set \(\Omega'\Subset\Omega\) of class \(C^1\), 
the trace of the normal component of \(\A\) on \(\partial\Omega'\) 
can be defined in the sense of distributions by
\begin{equation*}
	\pscal{\Trace[]{\A}{\partial\Omega'}}{\varphi}
	:= \int_{\Omega'} \A\cdot \nabla\varphi\, dx + \int_{\Omega'} \varphi\, d\Div\A,
	\qquad
	\forall\varphi\in C^\infty_c(\Omega).
\end{equation*}
It can be proved (see \cite{AmbCriMan}) that this distribution is
induced by an \(L^\infty\) function on 
\(\partial\Omega'\), still denoted by \(\Trace[]{\A}{\partial\Omega'}\), satisfying
\begin{equation}\label{f:infestrace}
	\|\Trace[]{\A}{\partial\Omega'}\|_{L^\infty(\partial\Omega'; \Haus\res \partial\Omega')}
	\leq \|\A\|_{L^\infty(\Omega')}.
\end{equation}
We fix an orientation on \(\Sigma\) by setting
\[
\nu_{\Sigma}(x) := \nu_{\Sigma_i}(x), \qquad \hh-\text{a.e. on}\  N_i
\]
and we define the weak normal traces of \(\A\) on \(\Sigma\)  by
\[
\Trm{\A}{\Sigma} := \Tr(\A, \partial\Omega_i),
\quad
\Trp{\A}{\Sigma} := -\Tr(\A, \partial\Omega'_i),
\qquad
\hh-\text{a.e.\ on}\ N_i.
\]
By \cite[Proposition 3.2]{AmbCriMan},
these definitions are independent of the choice
of $\Sigma_i$ and $N_i$. 

Moreover, the normal traces belong to
\(L^{\infty}(\Sigma, \H^{N-1}\res\Sigma)\) 
and it holds that
\begin{equation}\label{f:trA}
	\Div \A \res\Sigma =
	\left[\Trp{\A}{\Sigma} - \Trm{\A}{\Sigma}\right]
	\, {\mathcal H}^{N-1} \res\Sigma
\end{equation}
(see \cite[Proposition~3.4]{AmbCriMan}).
In particular, by \eqref{f:infestrace},
$|\Div\A|(\Sigma) \leq 2 \|\A\|_\infty \Haus(\Sigma)$.

In what follows, we use the notation 
\begin{equation}\label{f:trstar}
\Trace[*]{\A}{\Sigma}:= \frac{\Trp{\A}{\Sigma} + \Trm{\A}{\Sigma}}{2}.
\end{equation}

\begin{remark}
A simple example that helps illustrate the above notation is the following.
Let \(E\) be an open, bounded set of class \(C^1\), and assume that
its boundary $\Sigma := \partial E$ is oriented by
the \textbf{inward} unit normal vector $\widetilde{\nu}_E$.
Let \(\A\) be a piecewise $C^1$ vector field
that extends continuously to vector fields \(\A_i\) and \(\A_e\)
in \(\overline{E}\) and \(\R^N\setminus E\) respectively.
Then
\[
\Trm{A}{\partial E} = \A_e \cdot \widetilde{\nu}_E,
\qquad
\Trp{A}{\partial E} = \A_i \cdot \widetilde{\nu}_E,
\]
and
\[
\Div \A \res\Sigma =
(\A_i - \A_e) \cdot \widetilde{\nu}_E
	\, {\mathcal H}^{N-1} \res\Sigma.
\]
\end{remark}

\section{Assumptions on the vector fields and known results}
\label{s:ipo}

In line with \cite{CD4}, this section details the assumptions on the vector fields and reviews the results required in what follows. Specifically, we consider  \(\bsmall\colon \Omega\times \R \to \R^N\) fulfilling the following requirements:
\begin{itemize}
	\setlength\itemsep{4pt}
	\item[(b1)]
	\(\bsmall\) is a bounded Borel function in $ \Omega\times \R$,
that is, $\|\bsmall\|_\infty := \|\bsmall\|_{L^\infty(\Omega\times\R, \R^N)} < +\infty$.
	
	\item[(b2)]
For \(\LLN\)--a.e.\ \(x\in\Omega\), the map
	\(t \mapsto\bsmall(x, t)\) is continuous in \(\R\). 
	
	\item[(b3)]
For every \(t\in\R\), the map
	\(x \mapsto\bsmall(x, t)\) belongs to $\DM(\Omega)$.
	
	\item[(b4)]
	The least upper bound of the measures 
 	$|\Div_x \bsmall(\cdot,t)|$,	
\[\sigma\defeq \bigvee_{t\in \R} |\Div_x \bsmall(\cdot,t)|\,,\]
is a Radon measure.
	
\end{itemize}

Recall that (see \cite[Definition~1.68]{AFP})
\[
\bigvee_{t\in \R} |\Div_x \bsmall(\cdot,t)| (E)=\sup \left\{ \sum_{t\in T}|\Div_x \bsmall(\cdot,t)|(E_t) \right\}
\]
where $T$ is a finite or countable subset of  $\R$, and  $(E_t)_{t\in T}$ is a partition of $E$.

\begin{remark}
In (b1), it is sufficient to assume that
\(\bsmall\) is a \textsl{locally} bounded Borel function,
but at the price of some additional technicalities in the proofs.
\end{remark}

\begin{remark}\label{r:sigmah}
By (b3) and Proposition \ref{p:basicVF}(i), we have 
$\Div_x \bsmall(\cdot, t) \in\radonh$ for every $t\in\R$.
Consequently, the measure $\sigma$ also belongs to $\radonh$.
\end{remark}

\begin{remark}\label{r:sigmah2}
From (b1) and Proposition 2.3 in \cite{CD4}, it follows that
there exists a set \(Z_1\subset\Omega\) such that \(\LLN(Z_1) = 0\) and,
for every \(t\in\R\), the function \(x\mapsto \bsmall(x, t)\)
is approximately continuous on \(\R^N\setminus Z_1\). 
\end{remark}

By the definition of the least upper bound of measures, we have
\(\Div_x\bsmall(\cdot, t) \ll \sigma\) for every \(t\in\R\).
We denote by \(\f(\cdot, t)\) the Radon--Nikod\'ym derivative
of \(\Div_x\bsmall(\cdot, t)\) with respect to \(\sigma\), so that for every $t\in\R$ it follows that
\[
\Div_x\bsmall(x, t) = \f(x, t)\, \sigma \ \text{for $\sigma$-a.e. } x\in\Omega,
\qquad
\f(\cdot, t) \in L^1(\Omega;\sigma).
\]
Furthermore, since \(|\Div_x \bsmall(\cdot, t)| \leq \sigma\), it follows that
\begin{equation*}
	\forall t\in\R:
	\quad
	|\f(x,t)| \leq 1 \quad
	\text{for \(\sigma\)-a.e.}\ x\in\Omega.
\end{equation*}

We extend $\bsmall$ by setting $\bsmall = 0$ in $(\R^N\setminus\Omega)\times\R$,
so that the vector field
\begin{equation}\label{f:B}
	\B(x,t) := \int_0^t \bsmall(x,s)\, ds,
	\qquad x\in\R^N,\ t\in\R,
\end{equation}
is defined for all $(x,t)\in\R^N\times\R$ and
 \(\B(x,0) = 0\) for every \(x\in\R^N\).
By (b2), for every \(x\in\R^N\)
we have \(\bsmall(x,t) = \partial_t \B(x,t)\) for every \(t\in\R\).
Moreover, by (b1) and $\B(x,0)=0$,
there exists a constant $C>0$ such that
\begin{equation}\label{f:LipB}
|\B(x,t)|\leq C |t|\ \ \ \qquad \forall x\in \R^N, \ t\in \R.
\end{equation}

\begin{remark}
In what follows, when we wish to emphasize the fact that we are dealing with a family of $\DM$ fields parametrized by $t$,
we will use the notation 
\[
\bsmall_t:=\bsmall(\cdot,t), \quad
\B_t:=\B(\cdot,t),  \qquad t\in\R.
\]
\end{remark}

The following properties of $\B$ are established in \cite[Lemmas 3.1 and 3.2]{CD4}.

\begin{lemma}\label{l:B}
	For every \(t\in\R\), we have
	\(\B(\cdot, t)\in\DM(\Omega)\) and \(\Div_x\B(\cdot, t) \ll \sigma\).
	If \(\F(\cdot, t)\) denotes the Radon--Nikod\'ym derivative
	of \(\Div_x\B(\cdot, t)\) with respect to \(\sigma\), we have that
	\begin{equation}\label{f:F}
		\Div_x\B(\cdot, t)= \F(\cdot,t) \, \sigma, \qquad 
		\F(x,t) = \int_0^t \f(x, s)\, ds,
		\qquad \text{for \(\sigma\)--a.e.}\ x\in\Omega.
	\end{equation}
	In particular, 	for \(\sigma\)--a.e.\ $x\in\Omega$,
	\begin{equation}\label{f:estiv}
		|\F(x,t) - \F(x,s)| \leq |t-s|, 
		\qquad
	\forall t,s\in\R.
	\end{equation}
\end{lemma}

\begin{lemma}\label{l:lebF}
	There exists a set \(Z_2\subset\Omega\)
	with \(\sigma(Z_2) = 0\) and \(\H^{N-1}(Z_2) = 0\) such that
	every \(x_0\in\Omega\setminus Z_2\)
	is a Lebesgue point of \(F(\cdot, t)\) 
	with respect to \(\sigma\),
	for every \(t\in\R\).	
\end{lemma}

\begin{remark}
	Since both $\Div_x\bsmall(x,t)$ and $\sigma$ are finite Radon measures, we may assume that $\f(\cdot,t)$, and hence $\F(\cdot,t)$,
	is a Borel function in $\Omega$ (see, e.g. \cite[Theorem 5.8]{Maggi}).
\end{remark}

One of the key properties of the fields considered here is that the composition $\vv(x) := \B(x,u(x))$, with 
$u\in\BVL[\Omega]$, is itself a $\DM$ field (see  \cite[Lemma~3.7]{CD4}).

\begin{lemma}\label{l:divv}
	Let $\bsmall$
	satisfy assumptions $(b1)-(b4)$,
	and let \(\B\) be defined by \eqref{f:B}.
	For every $u\in \BVL[\Omega]$,
	the function 
%$\vv\colon\Omega\rightarrow\R^N$, defined by
	\[
	\vv(x):=\B(x,u(x))
	\,,
	\qquad x\in\Omega,
	\]
	belongs to $\DM(\Omega)$, and satisfies
	\begin{equation*}
		|\Div\vv|(K) \leq 
		\|u\|_{L^\infty(K)} \sigma(K) +
		\|\bsmall\|_\infty |Du|(K)
		\qquad \forall K\Subset\Omega,\ K\ \text{compact}.
	\end{equation*}
	Moreover, if
	$\B_\varepsilon(\cdot, t) := \rho_\varepsilon \ast \B(\cdot, t)$
	denotes the standard mollification of $\B(\cdot, t)$, then
	\(\vv_\varepsilon (x) := \B_\varepsilon(x, u(x))\)
	converges to \(\vv\) 
	a.e.\ in \(\Omega\) and
	in \(L^1(\Omega, \R^N)\).
\end{lemma}

\begin{remark}
	In what follows, we will use the notation
	\[
	\Div_x \B(x, u(x)) :=
	(\Div_x \B) (x, u(x)) =
	\left.\Div_x \B (x, t)\right|_{t = u(x)}. 
	\]
	The divergence of the composite function $x\mapsto \B(x,u(x))$ will be denoted simply by
	$\Div \left(\B(x, u(x))\right)$.
\end{remark}

We introduce here three classes of fields satisfying (b1)--(b4), for which the subsequent results simplify. 

\begin{example}[Autonomous case]
If $\bsmall(x,t)=\bsmall_0(t)$ with $\bsmall_0\in C_b(\R,\R^N)$, then $\sigma=0$, 
and $\B$ is the $C^1$ integral function of $\bsmall_0$.
\end{example}

\begin{example}[Separated fields]
\label{ex:separated}
	If $\bsmall(x,t)=g(t)\A(x)$ with $g\in C_b(\R)$ and $\A\in\DM(\Omega)$, then 
	\[
	\sigma=\|g\|_{L^\infty(\R)}\, |\Div \A|,\qquad  f(x,t)= g(t)\|g\|_{L^\infty(\R)}^{-1},
	\] 
	and $\B(x,t)=\A(x)G(t)$, where $G$ is the $C^1$ integral function
	\[
	G(t)=\int_0^t g(s)\, ds,
	\qquad x\in\R^N,\ t\in\R.
	\]
	The linear case treated in earlier works \cite{Anz,ChenFrid,CCDM,CD3,CDM} is recovered for $g(t)=1$.
\end{example}

\begin{example}[Separable fields]
\label{ex:splitted}
	Another interesting case, which we refer to as the \textsl{separable} case, occurs when
	the domain admits a decomposition $\Omega=\Omega^+ \cup \Omega^-$, 
	with $\Omega^+\cap \Omega^- =\emptyset$, such that, for all $t\in\R$,
	\begin{equation*}
		\begin{split}
			\Div_x \bsmall(\cdot,t) \geq 0\quad  \text{in}\ \Omega^+, \qquad
			\Div_x \bsmall(\cdot,t) \leq 0\quad  \text{in}\ \Omega^-,
		\end{split}
	\end{equation*}
equivalently,
	\begin{equation}\label{splitomega1}
		\begin{split}
			\f(\cdot, t) \geq 0\quad  \sigma-\text{a.e.\ in}\ \Omega^+, \qquad
			\f(\cdot, t) &\leq 0\quad  \sigma-\text{a.e.\ in}\ \Omega^- ,
		\end{split}
	\end{equation}
	and, for $\sigma$--a.e.\ $x\in \Omega^+$, the map
$t\mapsto F(x,t)$ is nondecreasing,
while for $\sigma$--a.e.\ $x\in \Omega^-$ it is 
nonincreasing.
\end{example}

\begin{remark}
Clearly, if $\bsmall$ is separated and $g$ does not change sign, then $\bsmall$ is separable with $\Omega^\pm$ given by the Hahn's decomposition of $\diver\A$.

Nevertheless, it is easy to construct separable fields that are not separated. 
For example, take 
$\Omega\subseteq \R^2$, $\bsmall(x_1,x_2,t)= \left(g_1(t), h(x_1,x_2)\, g_2(t) \right)$, with 
$g_1, g_2 \in C_b(\R)$, $g_1 \neq g_2$, $g_2 \geq 0$.
In this case, 
$\Div_x\bsmall(x_1,x_2,t)=h_{x_2}(x_1,x_2) g_2(t)$, and $\Omega^\pm$ are determined by the sign of $h_{x_2}$.
\end{remark}

By Lemma~\ref{l:B} and the results recalled in subsection \ref{ss:div},
 it follows that for every \(t\in\R\),
the traces of the normal component of the vector field \(\B_t\)
on an oriented countably \(\H^{N-1}\)--rectifiable set \(\Sigma\subset\Omega\)
can be defined in the sense of distributions, denoted by
\(\Trace[i,e]{\B_t}{\Sigma}\).
These distributions are induced by \(L^\infty\) functions on \(\Sigma\),
still denoted by \(\Trace[i,e]{\B_t}{\Sigma}\), and satisfy
\[
\|\Trace[i,e]{\B_t}{\Sigma}\|_{L^\infty(\Sigma, \H^{N-1})}
\leq \|\B_t\|_{L^\infty(\Omega)}.
\]
In what follows we use the notation
\begin{equation}\label{f:beta}
	\beta^i(\cdot, t) := \Trp{\B_t}{\Sigma},
	\quad
	\beta^e(\cdot, t) := \Trm{\B_t}{\Sigma},
\end{equation}
so that, by \eqref{f:trA}, 
\begin{equation}\label{mmm}
	\Div_x \B(\cdot, t) \res\Sigma =
	\left[\beta^{i}(\cdot,t)-\beta^{e}(\cdot,t)\right]\, {\mathcal H}^{N-1} \res\Sigma\,.
\end{equation}

The basic estimates for the weak normal traces are established in \cite[Lemma 3.3 and Proposition 3.4]{CD4}.

\begin{proposition}
\label{p:lip}
Let \(\Sigma\subset\Omega\) be an oriented countably \(\H^{N-1}\)--rectifiable set.
There exists a subset \(\Sigma'\subseteq\Sigma\)
with \(\H^{N-1}(\Sigma\setminus\Sigma') = 0\),
and representatives \(\overline{\beta}^{i,e}\) of
\(\beta^{i,e}\)
(in the sense that, for every \(t\in\R\),
the chosen \(\overline{\beta}^{i,e}(\cdot, t)\)
agrees with
\(\beta^{i,e}(\cdot, t)\) \(\H^{N-1}\)--a.e.\ in \(\Sigma\))
such that
\begin{equation*}%\label{f:lip}
	|\overline{\beta}^{i,e}(x,t)
	- \overline{\beta}^{i,e}(x,s)|
	\leq \|\bsmall\|_\infty\, |t-s|
	\qquad
	\forall x\in\Sigma',\ t,s\in\R.
\end{equation*}
\end{proposition}

\begin{remark}
	In what follows, we always identify \(\beta^{i,e}\) with 
	the representatives \(\overline{\beta}^{i,e}\)
	provided by Proposition~\ref{p:lip}.
\end{remark}

\begin{remark}\label{r:trbsmall}
Since $\bsmall_t\in \DM(\Omega)$ for every $t\in\R$, the normal traces
\[
\gamma^{i,e}(\cdot, t):=\Trace[i,e]{\bsmall_t}{\Sigma}
\]
belongs to $L^\infty(\Sigma)$.
Moreover, one easily checks that
$$
\beta^{i,e}(\cdot, t)=\int_0^t\gamma^{i,e}(\cdot, s)\,ds.
$$
\end{remark}

Finally, we recall the representation formula for the weak normal traces of the composite function
$\B(x,u(x))$, as given in \cite[Proposition 4.5]{CD4}.

\begin{proposition}\label{p:traces}
Let $\bsmall$
satisfy assumptions $(b1)$-$(b4)$, 
let $u\in\BV\cap L^\infty(\Omega)$
and let $\Sigma\subset\Omega$ be an oriented countably $\H^{N-1}$-rectifiable set.
%oriented as in Section~\ref{ss:div}.
Then the weak normal traces on $\Sigma$ of the composite function $\B(x,u(x))\in\DM(\Omega)$, 
as defined in Lemma \ref{l:divv}, are given by
\begin{gather}
\label{f:tracesv1}
\Trp{\B(\cdot,u)}{\Sigma} (x)=
\begin{cases}
	\beta^{i}(x, u^i(x)),
	&\text{for $\H^{N-1}$-a.e.}\ x\in J_u,\\
	\beta^{i}(x, \ut(x)),
	&\text{for $\H^{N-1}$-a.e.}\ x\in \Sigma\setminus J_u,
\end{cases}
\\ \label{f:tracesv2}
\Trm{\B(\cdot,u)}{\Sigma}(x)=
\begin{cases}
	\beta^{e}(x, u^e(x)),
	&\text{for $\H^{N-1}$-a.e.}\ x\in J_u,\\
	\beta^{e}(x, \ut(x)),
	&\text{for $\H^{N-1}$-a.e.}\ x\in \Sigma\setminus J_u.
\end{cases}
\end{gather}
With the convention $u^{i,e}(x) = \ut(x)$ for $x\in\Omega\setminus S_u$, we obtain the compact expression
\[
\Trace[i,e]{\B(\cdot,u)}{\Sigma} (x)= \beta^{i,e}(x, u^{i,e}(x))
\qquad
\text{for $\H^{N-1}$-a.e.\ $x\in\Sigma$}.
\]
\end{proposition}

\section{Definition and basic properties of nonlinear pairings}\label{s:pair}

We introduce a new definition of pairing between a $BV$ function  $u$ and the field $\bsmall(x,u(x))$, depending on the selected representative $\Prec{u}$ (see \eqref{f:pr}). 
Relying on Lemmas~\ref{l:B} and~\ref{l:divv}, this pairing extends the divergence chain rule \eqref{classical} to the present setting, providing a well--defined measure--theoretic interpretation.

\begin{definition}[Nonlinear $\lambda$-pairing]\label{d:pair}
Let $\bsmall$
satisfy assumptions $(b1)-(b4)$, 
let \(\B\) be defined by \eqref{f:B}, and let $u\in\BV\cap L^\infty(\Omega)$.
For 
$\lambda\in\spacelambda$,
the nonlinear $\lambda$--pairing
between $\bsmall(\cdot,u)$ and $Du$ is the 
finite Radon measure in $\Omega$ defined by 
\begin{equation}
	\label{f:pairlambda}
	\pair{\bsmall(\cdot,u), Du}:= -F(x, \Prec u)\, \sigma + \Div(\B(x,u))\,.
\end{equation}
\end{definition}

By Lemma \ref{l:B}, the definition of $\sigma$, and Lemma \ref{l:divv}, this measure belongs to $\radonh$.

\begin{remark}
Using \eqref{f:B} and \eqref{f:F}, we obtain,
for every $\varphi\in C_c^1(\Omega)$,
\begin{equation}
	\label{f:bxurepr1}
		\langle(\bsmall(\cdot,u), Du)_\lambda,\varphi\rangle  = 
		-\int_\Omega \varphi(x) \int_0^{u^\lambda(x)} f(x,t)\, dt\, d\sigma
	 - \int_\Omega \int_0^{u(x)}  \bsmall(x,t)\cdot \nabla\varphi(x)\, dt\, dx.
\end{equation}
\end{remark}

In \cite{CD4}, the first two authors introduced another nonlinear  pairing, defined by
\begin{equation}\label{f:bxu}
\spair{\bsmall(\cdot,u), Du} :=
-\frac{1}{2}
\left[ F(x, u^+(x)) + F(x, u^-(x)) \right]\, \sigma
+ \Div(\B(x,u))\,.
\end{equation}
We refer to  $\spair{\bsmall(\cdot,u), Du}$ as the 
\textsl{standard (external) nonlinear pairing}. 
Its extension to a general
$\lambda \not\equiv 1/2$ is given next.

\begin{definition}[External nonlinear  $\lambda$-pairing]\label{d:expair}
	Let $\bsmall$
	satisfy assumptions $(b1)-(b4)$, 
	let \(\B\) be defined by \eqref{f:B}, and let $u\in\BV\cap L^\infty(\Omega)$.
For $\lambda\in\spacelambda$,
	the external nonlinear $\lambda$--pairing
	between $\bsmall(\cdot,u)$ and $Du$ is the measure 
	\begin{equation}
		\label{f:pairlambdaext}
		\Pair{\lambda}{\bsmall(\cdot,u), Du}:= -\left((1-\lambda)F(x, u^-)+\lambda F(x, u^+)\right) \, \sigma + \Div(\B(x,u))\,.
	\end{equation}
\end{definition}

\begin{remark}
Since the nonlinear pairings differ from the external ones only on the jump set $J_u$, they share many of the same fundamental properties. 
In particular, for every $\lambda\in\spacelambda$,
\begin{equation}\label{nosalti}
\begin{split}
\pair{\bsmall(x,u),Du}\res (\Omega\setminus J_u) & =-F(x, \widetilde u)\, \sigma \res (\Omega\setminus J_u) + \Div \B (x,u) \res (\Omega\setminus J_u)  \\ &= 	\Pair{\lambda}{\bsmall(x,u),Du} \res (\Omega\setminus J_u).			
\end{split}
\end{equation}
However, the external pairing is not suitable for addressing semicontinuity issues (see Section \ref{s:sc}), while the Coarea Formula seems to hold in general only for this specific choice (see Remark \ref{r:nocoaera}). Hence we develop the theory for both pairings.
\end{remark}

\begin{remark}
\label{r:auto1}
In the autonomous case $\bsmall(x,t)=\bsmall(t)$, the pairing measure is independent of $\lambda$, since
$\sigma=0$, and
\begin{equation}\label{autonomous}
	(\bsmall(u), Du)_\lambda=[\bsmall(u), Du]_\lambda=\Div(\B(u)).
\end{equation}
\end{remark}

The following result exhibits an explicit representation of $\lambda$--pairings on the jump set $J_u$ of $u$, in terms of the weak normal traces of $\B$.

\begin{proposition}
\label{p:tracesonJ}
Let $u\in\BV\cap L^\infty(\Omega)$ be given, and assume 
that the jump set $J_u$ is oriented so that $u^i=u^+$ and $u^e=u^-$. Then
\begin{equation}
\label{f:tronJ}
\begin{split}
 &	\pair{\bsmall(\cdot,u), Du}  \res J_u  \\  &
 \phantom{bbbbbbb} =
	[(\beta^i(x,u^+)-\beta^i(x,\Prec{u}))+(\beta^e(x,\Prec{u})- \beta^e(x,u^-))] \H^{N-1} \res J_u \\ & 
	\phantom{bbbbbbb} =
	\left[\mean{u^-(x)}^{u^+(x)}
\Trl{\bsmall_t}{J_u}(x)
	\,dt\right] |D^ju|\,.
\end{split}
\end{equation}
Here 
$\Trl{\bsmall_t}{J_u}
:=\chi_{[u^-, u^\lambda]}(t)
\Trm{\bsmall_t}{J_u}
+\chi_{[u^\lambda, u^+]}(t)
\Trp{\bsmall_t}{J_u}$.
Moreover, 
\begin{equation}
	\label{f:tronJ666}
	\begin{split}
		&	\Pair{\lambda}{\bsmall(\cdot,u), Du}  \res J_u  \\  & \phantom{bbbbbbb} 
		=
		\left[(1-\lambda)[\beta^i(x,u^+)-\beta^i(x,u^-)]+\lambda[\beta^e(x,u^+)-\beta^e(x,u^-)]\right] \H^{N-1} \res J_u
		\\  & \phantom{bbbbbbb} 
		=\left[\mean{u^-(x)}^{u^+(x)}[(1-\lambda)\gamma^i(x,t)+\lambda\gamma^e(x,t)]\,dt\right] |D^ju|.
	\end{split}
\end{equation}
\end{proposition}

\begin{proof}
Recalling that $F(x,\Prec{u}(x)) \, \sigma = \Div_x \B(x,\Prec{u}(x))$ and using \eqref{mmm}, we get
\begin{equation*}%\label{f:trJ1}
	F(x,\Prec{u}(x)) \, \sigma \res J_u= \left[\beta^i(x,\Prec{u}(x)) - \beta^e(x,\Prec{u}(x)) \right]\H^{N-1} \res J_u.
\end{equation*}
On the other hand, since $\B(x,u(x))$ belongs to $\DM(\Omega)$, by \eqref{f:trA}
we have that
\begin{equation}
	\label{f:trJ2}
\begin{split}
\Div\B(x,u(x))\res J_u  & = \left[ \Trp{\B(x,u)}{J_u}- \Trm{\B(x,u)}{ J_u}\right] \H^{N-1} \res J_u \\
& = \left[ \beta^i(x,u^+(x)) - \beta^e(x,u^-(x))\right] \H^{N-1} \res J_u.
\end{split}
\end{equation}
where in the last equality we have used \eqref{f:tracesv1},  \eqref{f:tracesv2}, and the orientation chosen for $J_u$.

The first equality in \eqref{f:tronJ} then follows from the definition of $\pair{\bsmall(\cdot,u), Du} $. 
Moreover, by Remark \ref{r:trbsmall} we infer that
\begin{equation*}%\label{f:tronJ26}
	\begin{split}
		& \pair{\bsmall(\cdot,u), Du}  \res J_u   \\ & \phantom{bbbbbbb} =
		[\beta^e(x,\Prec{u}(x))- \beta^e(x,u^-(x))+\beta^i(x,u^+(x))-\beta^i(x,\Prec{u}(x))] \H^{N-1} \res J_u
		\\  & \phantom{bbbbbbb}
		=\left[\int_{u^-}^{u^\lambda}\gamma^e(x,t)\,dt+\int_{u^\lambda}^{u^+}\gamma^i(x,t)\,dt\right]
		\H^{N-1} \res J_u
		\\  & \phantom{bbbbbbb}
		=\left[\int_{u^-}^{u^+}[\chi_{[u^-, u^\lambda]}\gamma^e(x,t)+\chi_{[u^\lambda, u^+]}\gamma^i(x,t)]\,dt\right] \H^{N-1} \res J_u
		\\  & \phantom{bbbbbbb}
		=\left[\mean{u^-}^{u^+}[\chi_{[u^-, u^\lambda]}\gamma^e(x,t)+\chi_{[u^\lambda, u^+]}\gamma^i(x,t)]\,dt\right] |D^ju|.
	\end{split}
\end{equation*}
Similarly, \eqref{f:tronJ666} follows from \eqref{f:trJ2} and the fact that
\begin{equation*}%\label{f:trJ1666}
	\begin{split}
		[(1 & - \lambda)   F(x, u^-)  + \lambda  F(x, u^+)] \, \sigma \res J_u   
		  \\ &
		= \left[(1-\lambda)\left[\beta^i(x,u^-) - \beta^e(x,u^-)\right]  
		+
		\lambda\left[\beta^i(x,u^+) - \beta^e(x,u^+)
		\right]\right]\H^{N-1} \res J_u.
\qedhere	
\end{split}
\end{equation*}
\end{proof}

A key property of the pairings $\pair{\bsmall(\cdot,u), Du}$ and $\Pair{\lambda}{\bsmall(\cdot,u), Du}$ is that they are absolutely continuous with respect to the measure $|Du|$.

\begin{theorem}%[$\mathcal{DM}^{\infty}$-dependence and $u\in BV$]
	\label{chainb4}
	Let $\bsmall$
	satisfy assumptions $(b1)-(b4)$ and let $u\in\BV\cap L^\infty(\Omega)$.
	Then, for every 
$\lambda\in\spacelambda$ 
and for every Borel set $E\subset\Omega$,
\begin{gather}
		|\pair{\bsmall(\cdot,u), Du}|(E)\leq
		\|\bsmall\|_{\infty}\,
		|Du|(E),
\label{f:mubdd}
\\
		|\Pair{\lambda}{\bsmall(\cdot,u), Du}|(E)\leq
		\|\bsmall\|_{\infty}\,
		|Du|(E)
\,.
\label{f:mubdd2}	
\end{gather}
\end{theorem}

\begin{proof} 
To prove \eqref{f:mubdd}, we first establish the inequality for the standard nonlinear pairing \eqref{f:bxu}, giving an argument independent of the proof in \cite{CD4}.
We use a regularization method similar to that in \cite[Theorem~3.4]{DCFV2}. 

Let
\(\B_ \varepsilon(\cdot, t) := \rho_ \varepsilon \ast \B_t\)
and define \(\vv_\varepsilon(x) :=\B_ \varepsilon(x,u(x))\).
Since $\B_\varepsilon$ is $C^1$ in $(x,t)$, 
the $BV$ chain rule
(see \cite[Theorem 3.96]{AFP}) yields
\begin{equation}\label{alpha1}
	\begin{split}
		\int_{\Omega}  \nabla\varphi(x) & \cdot\vv_\varepsilon(x)\, dx =
		-\int_{\Omega} \varphi(x) \Div_x \B_ \varepsilon (x,u(x))\, dx
		\\ & -
		\int_{\Omega} \varphi(x)\, \bsmall_\varepsilon(x, \widetilde{u}(x))\, d D^du
\\ & -
			\int_{J_u} \varphi(x) \left(\B_ \varepsilon (x,u^+(x))- \B_ \varepsilon (x,u^-(x))\right)\cdot \nu_u\, 
			d\Haus,
	\end{split}
\end{equation}
for every $\varphi\in C^1_c(\Omega)$.
By Lemma \ref{l:divv},
\(\vv_\varepsilon(x) \to \vv(x)\) for $\mathcal L^N$--a.e.\ $x\in\Omega$, and by \eqref{f:LipB} we have
\[
\|\vv_\varepsilon\|_{L^\infty(\Omega, \R^N)} \leq \|\vv\|_{L^\infty(\Omega, \R^N)} \leq C \|u\|_{L^\infty(\Omega)}.
\]
Hence, by the Dominated Convergence Theorem,
\begin{equation*}%\label{f:alpha2}
	\lim_{ \varepsilon\to 0^+} \int_{\Omega} \nabla\varphi(x)\cdot \vv_ \varepsilon(x)\, dx 
=\int_{\Omega} \nabla\varphi(x)\cdot \vv(x)\, dx \,,
\end{equation*}
which means that
\begin{equation*}%\label{f:sinistra1}
\lim_{ \varepsilon\to 0^+}	\Div [\B_\varepsilon(x,u(x))]\,\mathcal L^N =
	\Div\vv (x)
	\,,\quad     \hbox{in the weak${}^*$ sense of measures. }
\end{equation*}
We now claim that
\begin{equation}\label{f:ie}
	\lim_{\varepsilon\to 0}I_\varepsilon  =
\frac{1}{2}	\int_{\Omega} \varphi(x) 	\left[ F(x, u^+(x)) + F(x, u^-(x)) \right]\, d\sigma(x), 
\end{equation}
where
\[ 
I_\varepsilon   := 
\int_{\Omega} \varphi(x) \Div_x \B_\varepsilon (x,u(x))\,dx,
\]
so that, passing to the limit in \eqref{alpha1}, we obtain
\begin{equation}\label{f:xstim}
	\begin{split}
	\langle\spair{\bsmall(\cdot,u),Du}\varphi\rangle  & = \lim_{ \varepsilon\to 0^+} 
\Big[	\int_{\Omega} \varphi(x)\, \bsmall_\varepsilon(x, \widetilde{u}(x))\, d D^du  \\ &
	+\int_{J_u} \varphi(x) \left(\B_ \varepsilon (x,u^+(x))- \B_ \varepsilon (x,u^-(x))\right)\cdot \nu_u\, 
	d\Haus \Big].
	\end{split}
\end{equation}

To prove \eqref{f:ie},
assume for simplicity that \(u\geq 0\), and choose \(C > \|u\|_\infty\).
By Fubini's Theorem, we have that 
\[
	\int_{\Omega} \varphi(x) \Div_x \B_\varepsilon (x,u(x))\,dx  =
	\int_{\Omega} \varphi(x) \int_0^{u(x)} \Div_x \bsmall_\varepsilon(x,t)\, dt \, dx\,.
\]
Since $0\leq u(x) \leq C$, we can rewrite the inner integral as
\[
\int_0^{u(x)} \Div_x \bsmall_\varepsilon(x,t)\, dt \, dx 
=
\int_0^{C}  \chi_{\{u > t\}}(x) \Div_x \bsmall_\varepsilon(x,t)\, dt\,.
\]
Thus
\[
\int_{\Omega} \varphi(x) \Div_x \B_\varepsilon (x,u(x))\,dx  =
\int_{\Omega} \varphi(x) \int_0^{C}  \chi_{\{u > t\}}(x) \Div_x \bsmall_\varepsilon(x,t)\, dt \, dx\,.
\]
Using the representation
\[
\Div_x \bsmall_\varepsilon(x,t)
=
\int_{\Omega}  \rho_\varepsilon(x-y)\, f(y,t)\, d \sigma(y)
\]
we obtain
\[
\begin{split}
\int_{\Omega} \varphi(x) \Div_x \B_\varepsilon (x,u(x))\,dx  
& =
\int_{\Omega} d\sigma(y)\, \int_0^{C} \,dt\,  f(y,t)
		\int_{\Omega} \varphi(x) \, \rho_\varepsilon(x-y)\, \chi_{\{u > t\}}(x)\, dx
\\ & =
\int_0^{C}	\int_{\Omega}  \rho_\varepsilon\ast (\varphi \chi_{\{u > t\}})(y)  f(y,t) \, d\sigma(y) \,dt\,.
\end{split}
\]
It is well known that for a.e.\ $t\in [0,C]$  the convolutions
\(\rho_\varepsilon\ast (\varphi \chi_{\{u > t\}})\) converge pointwise to
\(\varphi \chi^*_{\{u > t\}}\)
\(\H^{N-1}\)-a.e., and hence \(\sigma\)-a.e., in $\Omega$
(see \cite[Corollary 3.80]{AFP}).
By Lemma~\ref{funzcaracter},
\begin{equation}\label{fusco}
\chi^*_{\{u > t\}}(y)= \frac{1}{2}\left( \chi_{\{u ^+> t\}}(y)+\chi_{\{u ^-> t\}}(y)\right)
\qquad \H^{N-1}-a.e.\ y \in \Omega.
\end{equation}
Passing to the limit as \(\varepsilon\to 0\), and using \eqref{fusco} we obtain
\[
\begin{split}
&	\lim_{\varepsilon\to 0}  	\int_{\Omega} \varphi(x) \Div_x \B_\varepsilon (x,u(x))\,dx   =
	 \int_0^{C}\, dt  \int_{\Omega} \varphi(y) \chi^*_{\{u > t\}}(y)
	\, f(y,t)\, d\sigma(y)
	\\ &
	= 
 \int_{\Omega} \varphi(y) \, d\sigma(y) \int_0^{C}\chi^*_{\{u > t\}}(y)
\, f(y,t)\, dt
	%	\\ & 
	=
\frac{1}{2}	\int_{\Omega} \varphi(x) (F(x, u^+(x)) +F(x, u^-(x)) )\, d\sigma(x),
\end{split}
\]
which completes the proof of~\eqref{f:ie}
and hence \eqref{f:xstim} holds true.

Finally,
let $\varphi \in C^1_c(\Omega)$ with support contained in a compact set $K$, 
let $r_0 > 0$ be such that $K_{r_0} := K+\overline{B_{r_0}} \subset \Omega$ 
and pick $r\in (0,r_0)$ so that $|Du|(\partial K_r)=0$.
Using the bound 
\begin{equation*}
\left|\B_ \varepsilon (x,u^+(x))- \B_ \varepsilon (x,u^-(x))\right| =
\left|\int_{u^-(x)}^{u^+(x)} \bsmall_\varepsilon(x,t)\, dt \right| \leq 
\|\bsmall\|_{\infty} \left| u^+(x)- u^-(x)\right|
\end{equation*}
we obtain the estimate
\begin{equation}\label{f:approx1}
\begin{split}
\Big| & \int_{\Omega} \varphi\,   \bsmall_\varepsilon(x, \widetilde{u})\, d D^du  +  
\int_{J_u} \varphi(x) \left(\B_ \varepsilon (x,u^+)- \B_ \varepsilon (x,u^-)\right)\cdot \nu_u\, 
d\Haus\Big| \\  &
\leq 
\|\varphi\|_{\infty}\|\bsmall\|_{\infty} |D^du|(K_r)+ \|\varphi\|_{\infty}\|\bsmall\|_{\infty} \int_{J_u \cap K_r} \left| u^+(x)- u^-(x)\right| d\Haus \\ &
= \|\varphi\|_{\infty}\|\bsmall\|_{\infty} |Du|(K_r).
\end{split}
\end{equation}
Estimates \eqref{f:xstim} and \eqref{f:approx1} give
\begin{equation*}%\label{f:mubddst}
	|\spair{\bsmall(\cdot,u), Du}|(E)\leq
	\|\bsmall\|_{\infty}
	|Du|(E),
	\qquad
	\text{for every Borel set}\
	E\Subset\Omega\,.
\end{equation*}
This is exactly estimate~\eqref{f:mubdd} for the standard pairing.

Using \eqref{nosalti}, we obtain
\begin{equation}\label{f:laest1}
	|\pair{\bsmall(\cdot,u), Du}|\res(\Omega \setminus J_u)= 	|\spair{\bsmall(\cdot,u), Du}| \res(\Omega \setminus J_u)\leq
	\|\bsmall\|_{\infty}|Du|\res(\Omega \setminus J_u).
\end{equation}
Next, by Proposition \ref{p:tracesonJ}, Proposition \ref{p:lip}, choosing the orientation of $J_u$ such that $\uint=\upiu$
and $\uext=\umeno$, 
we deduce
\begin{equation}\label{f:laest2}
\begin{split}
		\left|\pair{\bsmall(\cdot,u), Du} \right|  \res J_u & \leq 
	\left|\beta^i(x,u^+(x))-\beta^i(x,\Prec{u}(x))\right|  \H^{N-1} \res J_u \\
	& +\left| \beta^e(x,\Prec{u}(x))- \beta^e(x,u^-(x))\right| \H^{N-1} \res J_u \\
	& \leq 	\|\bsmall\|_{L^\infty(K, \R^N)} (|u^+(x)- \Prec{u}(x) | + | \Prec{u}(x)-u^-(x)|) \H^{N-1} \res J_u 
	\\ & = \|\bsmall\|_{L^\infty(K, \R^N)} (u^+(x)-u^-(x)) \H^{N-1} \res J_u  \\
	& =
	\|\bsmall\|_{L^\infty(K, \R^N)}|Du | 	\res J_u.
\end{split}
\end{equation}
Combining \eqref{f:laest1} and \eqref{f:laest2},
we obtain estimate~\eqref{f:mubdd}. 

The proof of \eqref{f:mubdd2} is analogous, using~\eqref{f:laest1} together with~\eqref{f:tronJ666}.
\end{proof}

\begin{remark}\label{r:Theta}
	From \eqref{f:mubdd} and \eqref{f:mubdd2} it follows that  there exist two Borel functions
	$\theta_\lambda(\bsmall, u; \cdot)$ and $\theta^\lambda(\bsmall, u; \cdot) \in L^1(\Omega; |Du|)$
	such that
	\begin{equation*}%\label{f:Theta1}
		\theta_\lambda(\bsmall, u; \cdot)\, |Du|=\pair{\bsmall(\cdot,u), Du} \,, \quad
 		\theta^\lambda(\bsmall, u; \cdot)\, |Du|=\Pair{\lambda}{\bsmall(\cdot,u), Du}
		\qquad \text{$|Du|$--a.e.\ in $\Omega$.}
	\end{equation*}
By Proposition \ref{p:tracesonJ}, we already know that
\begin{gather*}
\theta_\lambda(\bsmall, u; x)  = \mean{u^-(x)}^{u^+(x)}
\Trl{\bsmall_t}{J_u}(x)\,dt,  
\qquad |D^ju|-\text{a.e. in}\ \Omega, \\
\theta^\lambda(\bsmall, u; x)  =
\mean{u^-(x)}^{u^+(x)}[(1-\lambda)\gamma^i(x,t)+\lambda\gamma^e(x,t)]\,dt, \qquad |D^ju|-\text{a.e. in}\ \Omega. 
\end{gather*}
Moreover, by \eqref{nosalti},
the two pairing measures $\Pair{\lambda}{\bsmall(\cdot,u), Du}$ and $(\bsmall(\cdot,u), Du)_\lambda$
have the same diffuse part, independent of $\lambda$, that we simply denote by $(\bsmall(\cdot,u), Du)^d$:
\[
(\bsmall(\cdot,u), Du)^d := \theta_\lambda(\bsmall, u; \cdot)\, 
|D^du|=\theta^\lambda(\bsmall, u; \cdot) \, |D^du|, \quad \lambda\in\spacelambda.
\]
Hence, also the corresponding densities coincide $|D^d u|$-a.e.\ in $\Omega$ and are independent of $\lambda$; 
we denote by \(\theta(\bsmall, u; \cdot) \in L^1(\Omega; |D^du|)\) their common value:
\begin{equation}\label{f:Theta}
\theta(\bsmall, u; \cdot) := \theta^\lambda(\bsmall, u; \cdot) = \theta_\lambda(\bsmall, u; \cdot),
\qquad  \text{$|D^du|$--a.e.\ in $\Omega$.}
\end{equation}
In \cite[Theorem 4.6]{CD4}, a representation is provided of the density of the standard nonlinear pairing:
\begin{equation}
\label{f:acp}
\theta(\bsmall, u; \cdot) \, |D^a u|=\Pair{*}{\bsmall(\cdot,u), Du}^a
=\left(\bsmall(\cdot,\widetilde u)\cdot \nabla u\right)\LLN,
\end{equation}
and hence it remains to represent the Cantor part $\theta(\bsmall, u; \cdot) \, |Du^c|$ of the pairing. 
This requires some additional effort, and it is proved in Theorem \ref{t:repr} below.
\end{remark}

We conclude this section showing that, in fact, the external $\lambda$--pairings can be identified with a subclass of the nonlinear pairings.
We start by proving a simple selection lemma, that will be useful also in the next sections.

\begin{lemma}
\label{l:meas}
Let $A\subset\R^N$ be a Borel set, and let $\psi\colon A\times [0,1]\to\R$
be a Carathéodory function, that is
\begin{itemize}
\item[(i)]
for every $s\in [0,1]$, the map $x\mapsto \psi(x,s)$ is Borel-measurable;
\item[(ii)]
for every $x\in A$, the map $s\mapsto \psi(x,s)$ is continuous.
\end{itemize}
Moreover, assume that, for every $x\in A$, the section
$\Gamma_x := \{s\in [0,1]\colon \psi(x,s) = 0\}$
is not empty.

Then, there exists a Borel-measurable function $\lambda\colon A\to [0,1]$ such that
$\psi(x, \lambda(x)) = 0$ for every $x\in A$.
\end{lemma}

\begin{proof}
By \cite[Lemma~6.4.6]{Boga}, the function $\psi$ is Borel-measurable in the
product space $A\times [0,1]$,
hence its zero-level set
\[
\Gamma := \{ (x,s)\in A\times [0,1]\colon
\psi(x,s) = 0\}
\]
is a Borel subset of $A\times [0,1]$.
For every $x\in A$, 
by assumption the section $\Gamma_x$ is not empty,
and, by the continuity of the function $s\mapsto \psi(x,s)$,
it is compact.

Hence, by \cite[Theorem~6.9.6]{Boga},
$\Gamma$ contains the graph of some 
Borel-measurable function $\lambda\colon A\to [0,1]$.
\end{proof}

\begin{proposition}\label{legame}
Let $u\in\BV\cap L^\infty(\Omega)$ be given.
Then, for every 
$\Lambda\in\spacelambda$ there exists $\lambda\in\spacelambda$,
depending on $u$ and $\Lambda$,
such that
\begin{equation}\label{laLa1}
\Pair{\Lambda}{\bsmall(\cdot,u), Du}=(\bsmall(\cdot,u), Du)_\lambda.
\end{equation}
Conversely, for every 
$\lambda\in\spacelambda$, there exists 
$\widetilde{\Lambda}\in\spacelambda$, 
depending on $u$ and $\lambda$,
such that
\begin{equation*}%\label{laLa2}
(\bsmall(\cdot,u), Du)_\lambda=\Pair{\widetilde\Lambda}{\bsmall(\cdot,u), Du}+R(u) \,\sigma\res J_u\,,
\end{equation*}
where
$R(u) (x) = 0$ if $F(x,u^{\lambda(x)})$ lies between $F(x, u^-(x))$ and $F(x, u^+(x))$
and, in general,
\[
|R(u)| \leq u^+ - u^-
\qquad\text{$\sigma$-a.e.\ on $J_u$} \,.
\]
\end{proposition}

\begin{proof}
Let $\Lambda\in\spacelambda$, 
let $A := \Omega\setminus (S_u\setminus J_u)$
and let
$\psi\colon A\times [0,1]\to\R$ be the function defined by
\[
\psi(x,s) :=
F(x, (1-s) u^-(x) + s u^+(x)) - (1-\Lambda(x)) F(x,u^-(x)) - \Lambda(x) F(x, u^+(x))\,.
\]
The function $\psi$ satisfies the assumptions (i)--(ii) of Lemma~\ref{l:meas}.
Moreover, by (ii) and the intermediate value theorem,
for every $x\in A$ the section $\Gamma_x$ is a non-empty subset of $[0,1]$.
Hence, by Lemma~\ref{l:meas}, there exists
$\lambda \in\spacelambda[A]$, 
that can be extended to a function $\lambda\in\spacelambda$,
so that
\[
\Lambda(x) F(x,\upiu(x))+(1-\Lambda(x))F(x,\umeno(x))=F(x,u^{\lambda(x)})
\qquad \forall x\in\Omega',
\]
and \eqref{laLa1} follows.

To prove the second part of the theorem,
it is enough to define
\[
\Lambda(x) := 
\begin{cases}
\frac{F(x,u^\lambda(x)) - F(x,u^-(x))}{F(x,u^+(x)) - F(x,u^-(x))}
& \text{if}\ F(x,u^-(x)) \neq F(x,u^+(x)),
\\
0
&\text{otherwise}
\end{cases}
\]
and $\widetilde{\Lambda} (x):= \Lambda(x) \wedge 1 \vee 0$.
(The choice $\Lambda(x) = 0$ when $F(x,u^-(x)) = F(x,u^+(x))$ is arbitrary and can be
replaced by any other Borel-measurable choice in $[0,1]$, since its value in this region
does not affect the pairing measure.)
Clearly, if $F(x,u^\lambda(x))$ lies between $F(x,u^-(x))$ and $F(x,u^+(x))$, then $\Lambda(x) \in [0,1]$ and $R(u)(x) = 0$.
Otherwise, it is enough to observe that $F$ is $1$-Lipschitz with respect to $u$, so that
$|R(u)(x)| \leq u^+(x) - u^-(x)$.
Specifically, consider for example a point $x\in J_u$ where
$F(x, u^-(x)) < F(x, u^+(x)) < F(x, \u^\lambda(x))$, so that
$\Lambda(x) > 1$ and $\widetilde{\Lambda}(x) = 1$. Then
\[
|R(u)(x)| = |F(x, \u^\lambda(x)) - F(x, u^+(x))|\leq
|u^\lambda(x) - u^+(x)| \leq u^+(x) - u^-(x)\,.
\]
Similarly, we can handle the other cases.
\end{proof}

\begin{remark}
The linear $\lambda$--pairing, introduced in \cite{CDM}, is recovered for $\bsmall(x,t)=\A(x)$, and
\[
(\A, Du)_\lambda=[\A, Du]_\lambda
\]
in $\Omega$, with the same $\lambda$. 
\end{remark}

\begin{remark}\label{r:leg}
If $t\mapsto F(x,t)$ is a monotone function on $\R$ for $\sigma$--a.e. $x\in \Omega$, 
then $R(u) \equiv 0$,
and hence Proposition \ref{legame}  provides a canonical identification between all external pairings and nonlinear pairings.

In particular, this is the case 
for the separated vector fields discussed in Example \ref{ex:separated}
(or even for the separable vector fields of Example \ref{ex:splitted}),
with $\bsmall(x,t)=g(t)\A(x)$, $g\in C_b(\R)$, $g\geq 0$, and $\A\in\DM(\Omega)$, for which
\[
(g(u)\A(x), Du)_\lambda=[g(u)\A(x), Du]_\Lambda
\] 
in $\Omega$, where the Borel functions $\lambda$ and $\Lambda$ are related by 
\[
G(u^\lambda(x))=\Lambda(x) G(u^+(x))+(1-\Lambda(x))G(u^-(x)),  \quad \text{$\sigma$--a.e.}\  x\in \Omega.
\]
A nonlinear pairing for separated fields $\B(x,t)=G(t) \A(x)$ can be defined exploiting the  Leibniz formula for linear 
$\lambda$--pairings (see \cite[Proposition 6.2]{CDM}): for  \(v\in\BVL[\Omega]\), and 
choosing on $J_u$ the orientation
such that  $\upiu=\uint$,
it holds that
\begin{gather*}
	\pair{v\A, Du}^d = \Prec{v}\pair{\A, Du}^d = v^* \spair{\A, Du}^d,
	%\label{f:gvADud}
	\\
	\pair{v\A, Du}^j 
	=  [(1-\lambda)\Trp{\A}{J_u}\vint+\lambda\Trm{\A}{J_u}\vext]\,
	(u^+ - u^-) \, \hh\res 
	J_u\,.
	%\label{f:gvADuj}
\end{gather*}
For $v=g(u)$ we get
\begin{gather*}
	\pair{g(u)\A, Du}^d =  g(u)^* \spair{\A, Du}^d,
	%\label{f:leib1},
	\\
	\pair{g(u)\A, Du}^j 
	=  [(1-\lambda)\Trp{\A}{J_u}g(u)^i+\lambda\Trm{\A}{J_u}g(u)^e]\,
	(u^+ - u^-) \, \hh\res 
	J_u\,.
	%\label{leib2}
\end{gather*}
We observe that, in general,
\(
\pair{g(u)\A, Du} \neq \pair{\A, DG(u)}.
\)
Specifically, from \eqref{f:tronJ} and observing that $J_{G(u)} \subseteq J_u$,
it holds that
\begin{gather*}
\pair{g(u)\A, Du}\res J_u = [G(u^\lambda) - G(u^-)] \Trm{\A}{J_u}
+  [G(u^+) - G(u^\lambda)] \Trp{\A}{J_u},
\\
\pair{\A, DG(u)}\res J_u = [(G(u))^\lambda - (G(u))^-] \Trm{\A}{J_u}
+  [(G(u))^+ - (G(u))^\lambda] \Trp{\A}{J_u},
\end{gather*}
but, generally,
$G(u^\lambda) \neq (G(u))^\lambda$ unless $g$ is constant. On the other hand, one can easily check that
\(
\Pair{\lambda}{g(u)\A, Du} = \Pair{\lambda}{\A, DG(u)}
\)
for every $\lambda \in\spacelambda$.
\end{remark}

\section{Coarea formula and slicing}
\label{s:coarea}

This section is devoted to the proof of the coarea formula
for the external nonlinear $\lambda$-pairing $\Pair{\lambda}{\bsmall(\cdot,u), Du}$,
and a related slicing result for the density $\theta^\lambda$ of its diffuse part.

We recall the notation $\bsmall_t(x):=\bsmall(x,t)$, so that 
$(\bsmall_t, Du)_\lambda=[\bsmall_t, Du]_\lambda$ is the linear $\lambda$--pairing between the field
$\bsmall_t \in \DM(\Omega)$ and the function $u\in\BVL[\Omega]$.

\begin{theorem}[Coarea formula]\label{t:coarea}
Let $\bsmall$ satisfy assumptions $(b1)-(b4)$. For $u\in \BV\cap L^\infty(\Omega)$, and
$\lambda \in\spacelambda$, 
it holds that
	\begin{equation}\label{f:gsplit}
		\pscal{\Pair{\lambda}{\bsmall(\cdot,u), Du}}{\varphi} = 
		\int_{\R} \pscal{\Pair{\lambda}{\bsmall_t, D\chiut}}{\varphi}\, dt,
\qquad \forall \varphi\in C_c^1(\Omega).
	\end{equation}
\end{theorem}

\begin{proof} %[Proof of Theorem \ref{t:coarea}]
By considering the level sets $\{u>t\}$ separately for
$t\geq 0$ and $t\leq 0$,
we may assume without loss of generality that \(u\geq 0\). 
Fixed $\varphi\in C_c^1(\Omega)$, we know that the distributional divergence of $\varphi \bsmall_t$ is a finite Radon measure in $\Omega$, and 
$$
\Div (\varphi\bsmall_t)=\varphi f(\cdot,t) \sigma+\bsmall_t \cdot\nabla\varphi\, \mathcal{L}^N.
$$
 Moreover $\chiut\in BV(\Omega)$ for $\LLU$-a.e.\ $t\in\R$, and 
hence for every $\lambda \in\spacelambda$ its precise representative $ \chi_{\{u > t\}}^{\lambda}$ belongs to
$L^1(\Omega; |\Div(\varphi \bsmall_t)|)$, and

$$
\pscal{ \Pair{\lambda}{\bsmall_t, D\chiut}}{\varphi}= -\int_\Omega \chi_{\{u > t\}}^{\lambda} \, d \Div (\varphi\bsmall_t).
$$	
By  Fubini's Theorem and recalling that, by Lemma \ref{funzcaracter}, for a.e. $t\in\R$,
\begin{equation}
\label{f:chil}
	\chi^\lambda_{\{u > t\}}= \lambda\chi_{\{u ^+> t\}}+(1-\lambda)\chi_{\{u ^-> t\}}
	\qquad \H^{N-1}-\text{a.e.\ in}\ \Omega,
\end{equation}
we get
\begin{align*}
	\int_{0}^{+\infty} & \pscal{ \Pair{\lambda}{\bsmall_t, D\chiut}}{\varphi}\, dt   =
	-\int_{0}^{+\infty} \left( \int_{\Omega} \chi_{\{u > t\}}^{\lambda} \, d \Div (\varphi  \bsmall_t)\right) \, dt	\\
	= {} & 
	- \int_{\Omega} \int_{0}^{+ \infty} (1 - \lambda) \chi_{\{u^- > t\}} \varphi f(x,t) \, dt\, d\sigma
	- \int_{\Omega} \int_{0}^{+ \infty} (1 - \lambda) \chi_{\{u^- > t\}} \bsmall_t \cdot\nabla\varphi\, dt \, dx\,  
\\
	&  - \int_{\Omega} \int_{0}^{+ \infty} \lambda \chi_{\{u^+ > t\}} \,  \varphi f(x,t)\,  dt \, \,d\sigma 
	- \int_{\Omega} \int_{0}^{+ \infty} \lambda \chi_{\{u^+ > t\}}  \,\bsmall_t \cdot\nabla\varphi\, dt\,dx \\
	= {} &
	 - \int_{\Omega} (1 - \lambda)  \varphi F(x, u^-)+ \lambda \varphi F(x, u^+) \,d\sigma 
	- \int_{\Omega}  \B (x,u)\cdot\nabla\varphi\,dx \\
	= {} & 
    \int_\Omega \varphi \, d \Pair{\lambda}{\bsmall(x,u), Du}.
\qedhere
\end{align*}	
\end{proof}

\begin{remark}\label{r:nocoaera}
By inspection of the proof, the validity of the coarea formula for the external pairing
depends on the linearity of $\chi^\lambda_{\{u > t\}}$ with respect to $\lambda$
(see \eqref{f:chil}).
It is then apparent that the coarea formula does not in general hold for the nonlinear pairing,
due to its intrinsic nonlinearity with respect to $\lambda$ on the jump set.
\end{remark}

\begin{remark}
	Localizing \eqref{f:gsplit} in $\Omega \setminus J_u$, and recalling that 
	$\pair{\bsmall(\cdot,u), Du}^d$ and $\Pair{\lambda}{\bsmall(\cdot,u), Du}^d$ coincide and do not depend on $\lambda$ in $\Omega \setminus J_u$,
	we obtain that
	\begin{equation}\label{f:gsplitnosalti}
		\begin{split}
		\int_{\Omega}\varphi\, &  d\pair{\bsmall(\cdot,u), Du}^d   = \int_{\Omega}\varphi\,  d\Pair{\lambda}{\bsmall(\cdot,u), Du}^d
		%\int_{\Omega\setminus J_u}\varphi\,  d(\bsmall(x,u), Du) =	
		\\ &
		=\int_{\R}\left(\int_{\Omega\setminus J_u}\varphi\, d \Pair{\lambda}{\bsmall_t, D\chiut}\right)\,dt
		\\& =
		-\int_{\R} \int_{\Omega\setminus J_u}  \chi_{\{\ut > t\}} \,\varphi\, f(x,t)\, d\sigma\, dt 
		- \int_{\R} \int_{\Omega\setminus J_u}   \chi_{\{\ut > t\}} \,\bsmall_t \cdot \nabla\varphi\, dx\, dt.
		\end{split}
	\end{equation}
\end{remark}

In order to provide a representation of the Cantor part of the $\lambda$--pairings,
we need some preliminary results.
We start by showing that the standard external pairing can be approximated by external pairings
with smooth vector fields.

\begin{theorem}[Approximation by $C^\infty$ fields] 
\label{t:convolu}
Let $\bsmall$
satisfy assumptions $(b1)-(b4)$,
and let $M>0$ be a given constant.
Then there exists a sequence of vector fields $\bsmall^k\colon\Omega\times\R\to\R^N$
satisfying, for every $k\in\N$:
\begin{itemize}
\item[(c1)] $\bsmall^k\in C_b(\Omega\times\R, \R^N)$,
with $\|\bsmall^k\|_\infty \leq \|\bsmall\|_\infty$.
\item[(c2)] For every $t\in\R$, the map $x\mapsto \bsmall^k(x,t)$
belongs to $C^\infty(\Omega, \R^N)$, and
\[
\int_\Omega |\Div_x\bsmall^k_t(x)|\, dx
\leq |\Div_x\bsmall_t|(\Omega) + \frac{1}{k}\,,
\qquad \forall t\in [-M,M].
\]
\item[(c3)]
The least upper bound
%$|\Div_x \bsmall^k(\cdot,t)|$,	
\[\sigma^k_M := \bigvee_{t\in [-M,M]} |\Div_x \bsmall^k_t|\Leb{N}\]
is a Radon measure,
with $\sigma^k_M(\Omega) \leq \sigma(\Omega) + \frac{1}{k}$.
\item[(c4)] For every $u\in\BV\cap L^\infty(\Omega)$,
with $\|u\|_\infty \leq M$, it holds that
\[
\Pair{*}{\bsmall^k(\cdot, u), Du} \stackrel{*}{\rightharpoonup}
\Pair{*}{\bsmall^{\phantom{k}}(\cdot, u), Du},
%\qquad \forall u\in\BV\cap L^\infty(\Omega),
\]
locally in the weak${}^*$ sense of measures in $\Omega$.
\end{itemize}
\end{theorem}

\begin{remark}
It is clear that a vector field satisfying (c1)--(c2)
also satisfies assumptions (b1)--(b3) given in Section~\ref{s:ipo}.
Property~(c3) is a localized version of~(b4),
with respect to the $t$ variable.
\end{remark}

\begin{proof}
The main arguments of the proof are taken
from~\cite[Theorem~1.2]{ChenFrid} and~\cite[Proposition~4.11]{CDM}.
In our setting some care is needed in order to handle
the $t$ dependence of the vector field.

We briefly recall the construction given in~\cite[Theorem 1.2]{ChenFrid} which, in turn, is similar to the analogous
construction for $BV$ functions
(see, e.g., \cite[Theorem~3.9]{AFP}).

Given $k\in\N$, $k\geq 1$,
since $\sigma$ is a finite Radon measure,
there exists $R>0$ such that
\[
\sigma(\Omega\setminus D_i) < \frac{1}{k}\,,
\qquad
i = 0, 1, 2, \ldots,
\] 
where
\[
D_i := \left\{x\in\Omega\colon \dist(x, \partial\Omega) > \frac{1}{R + i}
\right\}  \cap B_{R+i}(0)\,.
\]
Consider the sets $\Omega_1 := D_2$ and
$\Omega_i := D_{i+1} \setminus \overline{D}_{i-1}$, 
$i \geq 2$, and 
let $(\varphi_i)$ be a partition of unity subordinate 
to the covering $\{\Omega_i\}$.
Let $(\rho_\varepsilon)$ denote the standard family of mollifiers.
For every $i\in\N$ we can choose $\overline{\varepsilon}_i > 0$
such that
\[
\spt \rho_{\overline{\varepsilon}_i} \ast
(\bsmall_t \, \varphi_i) \subset
D_{i+2} \setminus \overline{D}_{i-2}\,,
\quad \forall t\in\R
\qquad (D_{-1} := \emptyset).
\]
Let us consider the map
\[
(0,\overline{\varepsilon}_i]\times [-M,M]
\ni (\varepsilon, t) \mapsto
\int_\Omega |\rho_\varepsilon \ast (\bsmall_t\, \varphi_i)
- \bsmall_t\, \varphi_i|\, dx,
\]
extended to $0$ for $\varepsilon = 0$.
Since $\bsmall\in L^\infty(\Omega\times\R, \R^N)$,
by the basic properties of convolution it is continuous
%hence uniformly continuous,
in $[0,\overline{\varepsilon}_i]\times [-M,M]$.
Hence, there exists $\varepsilon_i \in (0, \overline{\varepsilon}_i]$
such that
\begin{equation}
\label{f:epsi}
\int_\Omega \left|
\rho_{\varepsilon_i} \ast (\bsmall_t\cdot \varphi_i) - 
\bsmall_t\cdot\varphi_i
\right|\, dx\leq \frac{1}{4 k\, 2^i}
\,,
\qquad\forall t\in [-M,M].
\end{equation}	
Using a similar argument, by possibly choosing a smaller
value of $\varepsilon_i$, it also holds that
\begin{equation}
\label{f:epsii}
\int_\Omega \left|
\rho_{\varepsilon_i} \ast (\bsmall_t\cdot\nabla\varphi_i) - 
\bsmall_t\cdot \nabla \varphi_i
\right|\, dx\leq \frac{1}{4 k\, 2^i}
\,,
\qquad\forall t\in [-M,M].
\end{equation}	
Finally, let us define
\[
\bsmall^k(x,t) :=
\sum_{i=1}^\infty \rho_{\varepsilon_i} \ast
(\bsmall_t\, \varphi_i)\,.
\]
By repeating verbatim the same arguments as in the proof 
of~\cite[Theorem 1.2, p.~93]{ChenFrid},
it follows that
\begin{equation}
\label{f:estbk}
\int_\Omega |\bsmall^k_t - \bsmall_t|\, dx
< \frac{1}{4k}\,,
\quad
\int_\Omega |\Div_x \bsmall^k_t|\, dx
\leq |\Div_x \bsmall_t| (\Omega) + \frac{1}{k}
\qquad \forall t\in [-M,M],
\end{equation}
so that properties~(c1) and~(c2) hold.
By the definition of $\sigma$ we have that
$|\Div_x \bsmall_t| (\Omega) \leq \sigma(\Omega)$,
so that also~(c3) follows.

The first estimate in~\eqref{f:estbk},
together with \cite[Theorem~1.1]{ChenFrid},
implies that
\[
|\Div_x \bsmall_t|(\Omega)
\leq
\liminf_{k\to+\infty} \int_\Omega |\Div_x \bsmall^k_t|\, dx,
\qquad \forall t\in [-M, M].
\]
Taking into account the second estimate in~\eqref{f:estbk},
we thus conclude that
\[
\lim_{k\to+\infty} \int_\Omega |\Div_x \bsmall^k_t|\, dx
= |\Div_x \bsmall_t|(\Omega),
\qquad \forall t\in [-M, M].
\]

\medskip
Let us prove (c4).
By repeating the same argument as in the proof of
\cite[Proposition~4.11]{CDM},
for every $v\in\BVL[\Omega]$, %with $\|v\|_\infty \leq M$,
and for every $t\in [-M, M]$,
it holds that
\begin{equation}\label{mlml}
	\lim_{k\to +\infty}\int_{\Omega} v \, \varphi \, \Div\bsmall^k_t\, dx
	= \int_{\Omega} v^* \,\varphi \, d \Div\bsmall_t= \int_{\Omega} v^* \,\varphi f(x,t)\, d \sigma
	\qquad \forall \varphi\in C_c(\Omega)
\end{equation}
(see \cite{CDM}, formula (4.8)). 
	
Let us fix $u\in \BV\cap L^\infty(\Omega)$,
with $\|u\|_\infty \leq M$, and $\varphi\in C^1_c(\Omega)$. 
To simplify the notation, we assume without loss of generality
that $u\geq 0$.
By the representation formula~\eqref{f:bxurepr1} and Fubini's Theorem, we have that
\begin{equation*}
%\label{f:bxuk}
\begin{split}
&\left\langle  \Pair{*}{\bsmall^k(\cdot,u), Du},\varphi\right\rangle  = 
-\int_\Omega \varphi \int_0^{u(x)} \,\Div_x \bsmall^k_t\, dt\, dx
- \int_\Omega \int_0^{u(x)} \bsmall^k_t\cdot \nabla\varphi\, dt\, dx
\\  & \phantom{xx} =  
- \int_0^{\infty} \int_\Omega \chi_{\{u > t\}}\, \varphi\,\Div_x \bsmall^k_t\,dx\, dt
- \int_0^{\infty} \int_\Omega \chi_{\{u > t\}}\, \bsmall^k_t\cdot
\nabla\varphi\, dx \, dt
\\ & \phantom{xx} =: -I^k_1 - I^k_2.
\end{split}
\end{equation*}
For every $t\in\R$ such that  $\chi_{\{u > t\}}\in\BV$, by \eqref{mlml} with $v = \chi_{\{u > t\}}$ we 
deduce that, as $k\to\infty$,
\[
\zeta^k(t) := 
\int_\Omega \chi_{\{u > t\}}\, \varphi\,\Div_x \bsmall^k_t\,dx
\to
\int_\Omega \chi_{\{u > t\}}^* \,\varphi f(x,t)\, d \sigma
\,.
\]
Let $K\Subset \Omega$ denote the support of $\varphi$
and let $a := \|u\|_{L^\infty(K)}$.
Since, by (c2) and (c3),
\[
|\zeta^k(t)| \leq \chi_{[0,a]}(t) \, \|\varphi\|_\infty 
(\sigma(K)+1)\,,
\]
by the Dominated Convergence Theorem and by Lemma~\ref{funzcaracter} we deduce that
\begin{equation}
\label{f:Ik1}
\begin{split}
& \lim_{k\to \infty} I^k_1 =
\int_0^{\infty} \int_\Omega \chi_{\{u > t\}}^* \,\varphi f(x,t)\, d \sigma=\int_0^{\infty} \int_\Omega \frac{\chi_{\{u^+ > t\}}+\chi_{\{u^- > t\}}}2 \,\varphi f(x,t)\, d \sigma
\\
&=\int_\Omega\varphi\frac12\left[\int_0^{u^+}f(x,t)\, dt+\int_0^{u^-}f(x,t)\, dt\right]\,d\sigma
=\int_\Omega \varphi \frac{F(x, u^-)+  F(x, u^+)}2 \,d\sigma\,.
\end{split}
\end{equation}
Let us compute the limit of $I^k_2$.
Since, for every $t\in [-M,M]$, $\bsmall^k_t\to \bsmall_t$ in $L^1(\Omega,\R^N)$,
it holds that
\[
\psi^k(t) := \int_\Omega \chi_{\{u > t\}}\, \bsmall^k_t(x)\cdot
\nabla\varphi\, dx \to
\int_\Omega \chi_{\{u > t\}}\, \bsmall_t\cdot
\nabla\varphi\, dx
\qquad (k\to +\infty).
\]
Since $\|\bsmall^k\|_{\infty} \leq \|\bsmall\|_{\infty}$ for every $k\in\N$, it holds that
\[
|\psi^k(t)| \leq \|\bsmall\|_{\infty}\, \int_\Omega|\nabla\varphi|\, dx\,,
\qquad \forall t\geq 0,\ \forall k\in\N,
\]
and hence, by the Dominated Convergence Theorem,
\begin{equation}
\label{f:Ik2}
\lim_{k\to +\infty} I^k_2
=
\int_0^{\infty} \int_\Omega \chi_{\{u > t\}}\, \bsmall_t\cdot
\nabla\varphi\, dx \, dt
=
\int_\Omega \int_0^{u} \bsmall_t\cdot \nabla\varphi\, dt\, dx
\,.
\end{equation}
The conclusion now follows from \eqref{f:Ik1} and \eqref{f:Ik2}.
\end{proof}

The following slicing property of the density generalizes the one proved in \cite[Proposition 5.2]{CDM} for the linear $\lambda$--pairing.

\begin{theorem}\label{t:theta}
Let $\theta(\bsmall,u;\cdot)$ be the density defined in \eqref{f:Theta}.
	Then for a.e. $t\in\R$
	\begin{equation}
		\label{f:gthetaut}
		\theta(\bsmall,u;x) =
		\theta(\bsmall_t, \chiut; x)
		\quad
		\text{for \(|D\chiut|\)-a.e.}\ x\in\Omega\setminus J_u\,. 
	\end{equation}
\end{theorem}

\begin{proof} 
	The conclusion \eqref{f:gthetaut} follows once we prove that, for every $a,b\in\R$, it holds that
	\begin{equation}\label{f:xslicdens}
\int_a^b dt
\int_{\Omega\setminus J_u} \theta(\bsmall,u;x)\, \varphi\, d |D\chiut|
=
\int_a^b dt
\int_{\Omega\setminus J_u} \theta( \bsmall_t, \chiut; x)\, \varphi\, d |D\chiut|\,.
\end{equation}

For fixed real numbers $a < b$, let us consider the truncation 
function 
$g(s) := s \vee a \wedge b$
and let $v := g\circ u$. 
The function $v$ satisfies
	\begin{equation}\label{f:uv}
		\begin{gathered}
			\{u > t\} = \{v > t\},\quad
			D\chiut = D\chi_{\{v > t\}}\,,
			\qquad \forall t\in [a,b),
			\\
			D\chi_{\{v > t\}} = 0\,,
			\qquad \forall t\in\R\setminus [a,b),%t < a, \  t \geq b,
		\end{gathered}
	\end{equation}
that, together with 
	\[
	\frac{d Du}{d|Du|} = \frac{d D\chiut}{d|D\chiut|}\,,
	\qquad
	\text{$|D\chiut|$-a.e.\ in}\ \Omega
	\]
	(see \cite[\S 4.1.4, Theorem~2(i), \S 4.1.5, Theorem~3(ii)]{GMS1}),
imply that
	\[
	\frac{d Du}{d|Du|}
	= \frac{d Dv}{d|Dv|}
	\qquad
	\text{$|Dv|$-a.e.\ in}\ \Omega.
	\]	 
	Let $(\bsmall^k)$ be the sequence of
	smooth vector fields approximating $\bsmall$ as in
	Theorem~\ref{t:convolu}. 
	By \cite[Proposition~2.3]{Anz}, and since $J_v\subseteq J_u$, we have that
	\[
	\theta(\bsmall^k, u; x) =
	\bsmall^k(x,\widetilde u(x))\cdot\frac{d Du}{d|Du|}(x)
	=
	\bsmall^k(x,\widetilde v(x))\cdot\frac{d Dv}{d|Dv|}(x)
	=\theta(\bsmall^k, v; x)
	\
	\]
	$|Dv|$-a.e.\ in $\Omega\setminus J_v$.
From the chain rule formula in BV (see \cite[Theorem~3.99]{AFP})
it holds that
\(
|Dv| \res (\Omega\setminus J_v) =
g'(\ut) |Du| \res (\Omega\setminus J_u),
\)
so that
\[
\theta(\bsmall^k, v;x) |Dv| \res (\Omega\setminus J_v) =
\theta(\bsmall^k, u; x) g'(\ut) |Du| \res (\Omega\setminus J_u).
\]
Passing to the limit as $k\to +\infty$, we obtain that
\[
\theta(\bsmall, v;\cdot) |Dv| \res (\Omega\setminus J_v) =
 \theta(\bsmall, u; \cdot) |Du| g'(\ut) \res (\Omega\setminus J_u)
= \theta(\bsmall, u;\cdot) |Dv| \res (\Omega\setminus J_v),
\]
so that
\begin{equation}\label{f:thetauv}
	\theta( \bsmall, u; x) = \theta( \bsmall, v; x)
	\qquad
	\text{$|Dv|$-a.e.\ in}\ \Omega\setminus J_v.
\end{equation}

Given $\varphi\in C^1_c(\Omega)$,
let us compute the diffuse part $(\bsmall(\cdot,v), Dv)^d$ of the pairing $(\bsmall(\cdot,v), v; x)$ restricted to  $\Omega\setminus J_u$ (recall that $\Omega\setminus J_u\subseteq \Omega\setminus J_v$).
By the definition \eqref{f:Theta} of $\theta( \bsmall, v; x)$, 
equality \eqref{f:thetauv}, the coarea formula in BV
(see \cite[Theorem~3.40]{AFP})
and \eqref{f:uv} it holds that
\begin{equation}
\label{f:eqv1}
\begin{split}
%\langle(\bsmall(\cdot,v), Dv)^d\res (\Omega\setminus J_u),\varphi\rangle
\int_{\Omega\setminus J_u} &\varphi(x)\,d(\bsmall(\cdot,v), Dv)
=
\int_{\Omega\setminus J_u} \theta( \bsmall, v; x)\, \varphi(x)\, d |Dv|
\\ & =
\int_{\Omega\setminus J_u} \theta( \bsmall, u; x)\, \varphi(x)\, d |Dv|
%	\\ & =
=\int_a^b dt
\int_{\Omega\setminus J_u} \theta( \bsmall, u; x)\, \varphi(x)\, d |D\chiut|\,.
\end{split}
\end{equation}
On the other hand, by the coarea formula \eqref{f:gsplitnosalti}
and by \eqref{f:uv},
it holds that
\begin{equation}
\label{f:eqv2}
\begin{split}
& \int_{\Omega\setminus J_u} \varphi\,d(\bsmall(\cdot,v), Dv)
=
\int_\R \, dt
\int_{\Omega\setminus J_u} \varphi\,d\left.(\bsmall_t, D\chi_{\{v > t\}})\right.
\\ & =
\int_a^b\, dt
\int_{\Omega\setminus J_u} \varphi\,d\left.(\bsmall_t, D\chi_{\{u > t\}})\right.
=
\int_a^b dt
\int_{\Omega\setminus J_u} \theta( \bsmall_t, D\chiut, x)\, \varphi\, d |D\chiut|\,.
\end{split}
\end{equation}
By comparison of \eqref{f:eqv1} and \eqref{f:eqv2},
we obtain that \eqref{f:xslicdens} holds for every $a < b$,
so that~\eqref{f:gthetaut} follows.
\end{proof}

\begin{remark}
In general \eqref{f:gthetaut} does not extend to $J_u$.
Consider, for example, the case $\bsmall(x,t) = g(t) \A(x)$, with $g\in C_b(\R)$
and $\A\in C^1(\R^N, \R^N)$.
By \cite[Remark 4.10]{CDM} the density of the linear pairing for the field $\bsmall_t$ does not depend on $\lambda$ and it is given by
\[
\theta^\lambda(\bsmall_t, \chi_{\{u > t\}}; x) =
\begin{cases}
g(t) \A(x)\cdot \nu_u(x),
&\text{if}\ t \in [u^-(x), u^+(x)],
\\
0, &\text{otherwise}\,.
\end{cases}
\]
On the other hand, since   $\gamma^i(x,t) = \gamma^e(x,t) = g(t)\, \A(x)\cdot\nu_u(x)$ for $\Haus$-a.e.\ $x\in J_u$, 
by~\eqref{f:tronJ666} we obtain
\[
\theta^\lambda(\bsmall, u; x ) = \A(x)\cdot \nu_u(x) \mean{u^-(x)}^{u^+(x)} g(s)\, ds\,.
\]
If we choose $u = \chi_{B_1}$, observing that
\[
D\chi_{\{u > t\}} =
\begin{cases}
\Haus\res {\partial B_1},
&\text{if}\ t\in [0,1),
\\
0,
&\text{otherwise},
\end{cases}
\]
the equality $\theta^\lambda(\bsmall,u;x) =	\theta^\lambda(\bsmall_t, \chiut; x)$ holds if and only if 
for a.e.\ $t\in [0,1]$,
\[
g(t) = \mean{0}^{1} g(s)\, ds,
\qquad
\text{for $\Haus$-a.e.\ $x\in\partial B_1$ s.t.\ $\A(x)\cdot\nu(x)\neq 0$},
\]
that is for  $g$ constant in $[0,1]$.
\end{remark}

\section{Integral representation of the pairing measure}
\label{s:intrepr}

We now employ the results of the previous section to obtain the representation of the diffuse
part of the pairing.
Recall that we already know the density of the absolutely continuous part of the
pairing, see~\eqref{f:acp}.

\begin{theorem}\label{t:diff}
Let $u\in BV(\Omega) \cap L^{\infty}(\Omega)$.
Then there exists a Borel set $B\subset\Omega$ with the following properties:
\begin{itemize}
\item 
[(i)] $|D^d u| (\Omega\setminus B) = 0$;

\item
[(ii)] setting $E_t := \{u > t\}$, 
for every $x\in B\setminus  S_u$
it holds that $E_{\ut(x)}$ is a set of finite perimeter in $\Omega$ and
$x \in \partial^* E_{\ut(x)}$.
As a consequence, $x \in \partial^* E_{\ut(x)}$ for $|D^d u|$-a.e.\ $x\in\Omega$;

\item[(iii)]
for every $x\in B\setminus S_u$, denoting $t := \ut(x)$, it holds that
\[
\polar[](\bsmall, u; x) =
	\polar[](\bsmall_{t}, u; x)
	= \polar[](\bsmall_{t}, \chi_{E_{t}}; x)
	= \Trace[*]{\bsmall_{t}}{\partial^* E_{t} }(x)
\]
(recall definition \eqref{f:trstar} of the last term).
\end{itemize}
\end{theorem}

\begin{proof}
The line of proof is similar to the one of \cite[Theorem~3.12]{CCDM}.
For the sake of completeness we report here some details.
	
Let $Z\subset\R$ be the set such that
for every $t\in \R\setminus Z$ the following hold:
\begin{itemize}
	\item[(a)]
	$E_t := \{u > t\}$ is a set of finite perimeter in $\Omega$;
	\item[(b)]
	$\hh\Bigl(\{x\in \Omega \setminus S_u\colon \ut(x)=t \}\setminus
	\bigl((\Omega \setminus S_u)\cap\partial^*\{u>t\}\bigr)\Bigr)=0$;
	\item[(c)]
	$\displaystyle \polar[](\bsmall, u; x) 
= \polar[](\bsmall_t, u; x)
= \polar[](\bsmall_t, \chi_{E_t}; x)
	= \Trace[*]{\bsmall_t}{\partial^* E_t }(x)
	\quad
	\text{for $\hh$-a.e.}\ x\in \partial^* E_t\cap (\Omega\setminus J_u).
	$
\end{itemize}
By the coarea formula in $BV$,
by Theorem~\ref{t:theta}, 
Corollary~3.7 and Theorem~3.12 in \cite{CCDM} (with $\A=\bsmall_t$),
we have that $\LLU(Z) = 0$.

Since $\LLU(Z) = 0$,
by \cite[Proposition~3.92(a)(c)]{AFP}, we have that
\[
\nabla u = 0 \quad \text{$\LLN$-a.e.\ in}\ u^{-1}(Z) \, \text{ and } \, 
|D^c u| (\ut^{-1}(Z)) = 0.
\]
As a consequence,
for $|D^d u|$-a.e.\ $x$ we have that
$\ut(x) \in \R\setminus Z$,
i.e.\ $|D^d u|(\ut^{-1}(Z)) = 0$.

For every $t\in\R\setminus Z$, let $N_t\subset \partial^* E_t$
be a set such that
the following hold:
\begin{itemize}
	\item[(d)]
	$\ut(x) = t$ for every $x\in (\Omega \setminus S_u) \cap (\partial^* E_t \setminus N_t)$;
	\item[(e)]
	equality in (c) holds for every $x\in \partial^* E_t \setminus N_t$.
\end{itemize}
By (b) and (c), the set $N_t$ can be chosen of zero
$\hh$ measure.

We claim that
\begin{equation}
	\label{f:B1}
	|D^d u|(\Omega \setminus B) = 0,
	\qquad
	\text{where}\
	B := \bigcup_{t\in\R\setminus Z} ( \partial^* E_t \setminus N_t)\,.
\end{equation}
Specifically, 
since the sets $\partial^* E_t \setminus S_u$, $t\in\R\setminus Z$, are pairwise disjoint
(see \cite[p.~356]{GMS1}),
we have that
$
\partial^* E_t \cap ((\Omega \setminus S_u) \setminus B) = (\Omega \setminus S_u) \cap N_t$,
hence, by the coarea formula for $BV$ functions,
\begin{align*}
	|D^d u|(\Omega\setminus B) = |Du|((\Omega \setminus S_u)\setminus B)
	& = \int_{\R\setminus Z} \hh(\partial^* E_t \cap ((\Omega \setminus S_u) \setminus B))\, dt \\
	& \le \int_{\R\setminus Z} \hh(N_t)\, dt = 0.
\end{align*}

Finally, for every $x\in B \setminus S_u$
(hence, by \eqref{f:B1}, for $|D^d u|$-a.e.\ $x\in\Omega$),
we have that $x \in \partial^* E_{\ut(x)}$. 
\end{proof}

As a consequence of Theorems~\ref{t:theta} and~\ref{t:diff},
and \cite[Corollary~4.13]{CCDM},
we obtain the following result.

\begin{corollary}
\label{c:diff}
Let $x\in B$, where $B$ is the set defined in Theorem~\ref{t:diff},
let $t := \ut(x)$,
and assume that $\bsmall_t$ is approximately continuous at $x$. Then 
$\theta(\bsmall, u; x) = \widetilde{\bsmall_t}(x) \cdot \nu_u(x)$.
\end{corollary}

Summarizing the previous results, we obtain the complete representation of the pairing densities.

\begin{theorem}\label{t:repr}
	Let $\bsmall$ satisfy assumptions $(b1)-(b4)$,
	let $u\in BV(\Omega) \cap L^{\infty}(\Omega)$, and assume 
that the jump set $J_u$ is oriented so that $u^i=u^+$ and $u^e=u^-$.
	Then for $|D^j u|$-a.e.\ $x\in\Omega$,
	\begin{equation}\label{f:reprJ}
		\begin{split}
			\theta^\lambda(\bsmall, Du; x) &= \lambda\beta^i(x,u^+(x))+(1-\lambda)\beta^e(x,u^-(x))) \\ &
			=\mean{u^-}^{u^+}[\lambda\gamma^i(x,t)+(1-\lambda)\gamma^e(x,t)]\,dt
			=:\Trace[\lambda]{\bsmall(\cdot,u)}{J_u}(x)\,,
		\end{split}
	\end{equation}
	and
	\begin{equation}\label{f:reprJ1}
		\begin{split}
			\polar(\bsmall,Du;x) & = (\beta^i(x,u^+(x))-\beta^i(x,\Prec{u}(x)))+(\beta^e(x,\Prec{u}(x))- \beta^e(x,u^-(x)))\\
			& = \mean{u^-}^{u^+}[\chi_{[u^\lambda, u^+]}\gamma^i(x,t)+\chi_{[u^-, u^\lambda]}\gamma^e(x,t)]\,dt
			=:
			\Trl{\bsmall(\cdot, u)}{J_u}(x)\,.
		\end{split}
	\end{equation}
	Moreover, 
	for $|D^du|$-a.e.\ $x\in\Omega$ and denoting $t := \ut(x)$,
	\begin{equation}\label{f:reprC3}
\polar[](\bsmall, Du; x) = \polar[](\bsmall_t, D\chi_{\{u> t\}}; x) = 
		\Trace[*]{\bsmall_t}{\partial^* \{u> t\} }(x).
	\end{equation}
In particular, for $|D^au|$-a.e.\ $x\in\Omega$
\begin{equation}
\label{f:rompimento}
\polar[]{(\bsmall, u; x)} = {\bsmall}(x, u(x)) \cdot \nu_u(x).
\end{equation}
\end{theorem}

\begin{proof}
One can easily get that \eqref{f:reprJ} and \eqref{f:reprJ1} follow from Proposition \ref{p:tracesonJ}.
On the other hand, the proof of \eqref{f:reprC3} is a
direct consequence of Theorems~\ref{t:theta} and~\ref{t:diff}(iii).
Formula~\eqref{f:rompimento} is a consequence of Corollary~\ref{c:diff} and Remark~\ref{r:sigmah2}
(see also~\eqref{f:acp}).
\end{proof}

As a consequence of \eqref{f:reprC3},
for vector fields of bounded variation in $x$ we obtain the following result.

\begin{corollary}\label{c:shil}
Let $\bsmall$ satisfy assumptions $(b1)-(b4)$
and let $u\in BV(\Omega) \cap L^{\infty}(\Omega)$. 
Assume, in addition, that $\bsmall_t \in BV(\Omega, \R^N)$ for every $t\in\R$.
Then
\begin{equation}\label{f:thetareg2}
	\theta(\bsmall, u; x) = \bsmall^*(x,\widetilde u(x)) \cdot \nu_u(x),
	\qquad
	\text{for $|D^d u|$-a.e.}\ x\in \Omega.
\end{equation}
\end{corollary}

\begin{remark}
\label{r:auto}
In the autonomous case $\bsmall(x,t)=\bsmall(t)$ 
the representation formulas \eqref{f:reprJ}, \eqref{f:reprJ1} and \eqref{f:thetareg2} yield the well known Volpert's formulas, 
\begin{equation*}%\label{chainrule22}
(\bsmall(u),Du)=
%\Div (\B(u))=
\left[\mean{u^-(x)}^{u^+(x)}\bsmall(t)\cdot\nu_u(x)
\,dt\right] |Du|\,,
\end{equation*}
in the sense of measures, with the convention that $\mean{a}^{a}\bsmall(t)\,dt=\bsmall(a)$. Indeed,
$$
\gamma^{i,e}(x,t)=\Tr^{i,e}(\bsmall_t,J_u)=\bsmall(t)\cdot\nu_u(x).
$$
\end{remark}

\begin{remark}
Property \eqref{f:reprC3} connects the density of the nonlinear pairings to the linear pairing's density, making it possible to obtain further integral representations of the density via cylindrical averages or averages over half-balls
(see \cite[Section 4]{CCDM}).
\end{remark}

\section{The Gauss--Green formula}
\label{s:green}

Let $E\Subset\Omega$ be a set of finite perimeter, and denote by $\partial^* E$ its reduced boundary.
We will assume that the generalized normal vector on \(\partial^* E\) coincides
\(\hh\)-a.e.\ on \(\partial^* E\) with the measure--theoretic 
inward unit normal vector \(\nuint_E\) to \(E\).

The following result generalizes the one proved in \cite[Theorem 6.3]{CDM} for linear pairings.

\begin{theorem}[Gauss--Green formula]\label{t:GG}
	Let $\bsmall$
	satisfy assumptions $(b1)-(b4)$, let \(\B\) be defined by \eqref{f:B}, and let $F$ be as in \eqref{f:F}. 	
Given  a bounded set \(E\) with finite perimeter and compactly contained in $\Omega$, 
let $\beta^i(\cdot,t)$ and $\beta^e(\cdot,t)$ denote the %inner and outer 
normal traces of the field \(\B(\cdot,t)\) 
on $\partial^*E$ defined in \eqref{f:beta}.  
Let  $u\in BV(\Omega)\cap L^\infty(\Omega)$, and 
$\lambda \in \spacelambda$.
	Then the following Gauss--Green formulas hold:
	\begin{equation}
		\label{primaaa}
	\int_{E^1} \big(\lambda F(x, u^+) + (1-\lambda) F(x, u^-)\big) \, d\sigma
	+ \int_{E^1}[\bsmall(\cdot,u), Du]_\lambda
	=-\int_{\partial ^*E} \beta^i(x,u^i) \ d\mathcal H^{N-1}
	\end{equation}
	\begin{equation}
	\label{secondaa}
	\int_{E^1\cup\partial^*E}  \big(\lambda F(x, u^+) + (1-\lambda)F(x, u^-)\big) \, d\sigma
	 + \int_{E^1\cup\partial^*E}[\bsmall(\cdot,u), Du]_\lambda 
	 = -\int_{\partial^*E} \beta^e(x,u^e) \ d\mathcal H^{N-1}
	\end{equation}
	 \begin{equation}
	 	\label{terzaa}
	\int_{E^1} F(x, u^\lambda)  \, d\sigma
	+ \int_{E^1}(\bsmall(\cdot,u), Du)_\lambda
	=
	-\int_{\partial ^*E} \beta^i(x,u^i) \ d\mathcal H^{N-1}
\end{equation}
	\begin{equation}
	\label{quartaa}
	\int_{E^1\cup\partial^*E} F(x, u^\lambda)  \, d\sigma
	 + \int_{E^1\cup\partial^*E}(\bsmall(\cdot,u), Du)_\lambda  = -\int_{\partial ^*E} \beta^e(x,u^e) \ d\mathcal H^{N-1}
	\end{equation}
where $E^1$ is the measure theoretic interior of $E$.
\end{theorem}

The first two formulas concern the external pairings, whereas the following two concern the nonlinear pairings.
Moreover, formulas~\eqref{primaaa} and~\eqref{terzaa} are related to the measure-theoretic interior $E^1$ of the set $E$,
whereas formulas~\eqref{secondaa} and~\eqref{quartaa} are related to its measure-theoretic closure $E^1\cup\partial^* E$. 

\begin{proof}
We shall prove formulas \eqref{primaaa} and \eqref{secondaa} 
following the lines of the proof of \cite[Theorem 6.1]{CD4}, where $\lambda=\frac12$.

Since $E$ is compactly contained in $\Omega$, we may assume without loss of generality
that $\Omega = \R^N$.
By Theorem~\ref{chainb4}, the composite function $\vv(x) := \B(x, u(x))$
belongs to $\DM(\R^N)$.
Since $E$ is a bounded set of finite perimeter, the characteristic function
$\chi_E$ is a compactly supported $BV$ function, so that
$\Div (\chi_E \vv) (\R^N) = 0$
(see \cite[Lemma~3.1]{ComiPayne}).

Applying the Leibniz formula for the divergence
\(
\Div(w \A) = w^\lambda \Div \A + (\A, Dw)_\lambda.
\)
to $w = \chi_E\in BV(\R^N)\cap L^\infty(\R^N)$ and $\A = \vv\in \DM(\R^N)$,
it follows that
\begin{equation*}%\label{f:eqgg1}
0 = \int_{\R^N} \Div (\chi_E \vv) 
= \int_{\R^N} (\vv, D\chi_E)_\lambda+\int_{\R^N} \chi_E^\lambda \, d\Div \vv .
\end{equation*}
From Proposition~\ref{p:traces} and the representation formula~\eqref{f:reprJ} we get
\begin{equation}\label{f:eqgg2}
	\begin{split}
\int_{\R^N} \chi_E^\lambda \, d\Div \vv & =-(\vv, D\chi_E)_\lambda(\R^N)  \\ & = -\int_{\partial^* E} 
\big((1-\lambda)\beta^i(x, u^i) +\lambda \beta^e(x, u^e)\big)
\, d\H^{N-1}.
\end{split}
\end{equation}
On the other hand, we notice that
\begin{equation*}%\label{f:chilambda}
\chi_E^\lambda = (1- \lambda) \chi_{E^1} + \lambda \chi_{E^1 \cup \partial^* E} = \chi_{E^1} + \lambda \chi_{\partial^* E},
\end{equation*}
since $\chi_E^- = \chi_{E^1}$ and $\chi_{E}^+ = \chi_{E^1 \cup \partial^* E}$.
Using again 
Proposition~\ref{p:traces} 
and
\eqref{f:pairlambdaext} 
we obtain
\begin{equation}\label{f:eqgg3}
\begin{split}
\int_{\R^N} \chi_E^\lambda \, d\Div \vv = {} &
\int_{E^1}\,d\Div \vv  + \lambda\int_{\partial^* E} \,d\Div \vv 
\\ = {} &
\int_{E^1}\,d\Div (\B(x,u))  +\lambda \int_{\partial^* E} \,[\beta^i(x,u^i)-\beta^e(x,u^e)]\, d\mathcal H^{N-1}
\\ = {} &\int_{E^1} [(1-\lambda) F(x, u^-)+\lambda F(x, u^+) ]\, d\sigma
+ \int_{E^1}\,d[\bsmall(\cdot,u),Du]_\lambda\\  {} &
 +\lambda \int_{\partial^* E} \,[\beta^i(x,u^i)-\beta^e(x,u^e)]\, d\mathcal H^{N-1}
 \,.
\end{split}
\end{equation}
Formula \eqref{primaaa} now follows from \eqref{f:eqgg2} and \eqref{f:eqgg3}.
The proof of \eqref{secondaa} is entirely similar.

In order to prove~\eqref{terzaa} and~\eqref{quartaa}, it is sufficient to observe that by~\eqref{f:pairlambda}
we obtain that
\begin{equation*}
\begin{split}
\int_{\R^N} \chi_E^\lambda \, d\Div \vv = {} &
\int_{E^1}\,d\Div \vv  + \lambda\int_{\partial^* E} \,d\Div \vv 
\\ = {} &
\int_{E^1}\,d\Div (\B(x,u))  +\lambda \int_{\partial^* E} \,[\beta^i(x,u^i)-\beta^e(x,u^e)]\, d\mathcal H^{N-1}
\\ = {} &\int_{E^1} F(x, u^\lambda)\, d\sigma
+ \int_{E^1}\,d(\bsmall(\cdot,u),Du)_\lambda\\  {} &
 +\lambda \int_{\partial^* E} \,[\beta^e(x,u^e)-\beta^i(x,u^i)]\, d\mathcal H^{N-1}
 \,.
\qedhere
\end{split}
\end{equation*}
\end{proof}

\section{Semicontinuity results}
\label{s:sc}

In this section we consider the pairing
as a function in $BV$:
\[
\BV\cap L^\infty(\Omega)\ni u \mapsto \pair{\bsmall(\cdot,u),Du} \in \mathcal{M}(\Omega),
\]
and we characterize the selections $\lambda\in\spacelambda$ making this function lower semicontinuous 
along bounded sequences converging in $L^1$ with controlled pointwise values (see  Definition \ref{d:conv}).

In order to find the correct selections giving the lower semicontinuity of this function, we consider
$D := \{(a,b)\in\R^2\colon a\leq b\}$, and we define
\[
\widehat{\mathcal{F}}(x,a,b):= \max_{t\in [a,b]}F(x,t) 
\qquad (a,b)\in D.
\] 
Recalling the definition~\eqref{f:F} of $F$,
we immediately deduce that  $\widehat{\mathcal{F}}(x,0,0) = F(x,0) = 0$.
Moreover, using the Lipschitz property~\eqref{f:estiv},
the following elementary facts can be easily checked.
\begin{itemize}
	\item[(i)] For every $x \in\Omega$ and $b\in\R$, the map $a\mapsto \widehat{\mathcal{F}}(x,a,b)$ is monotone non-increasing in $(-\infty, b]$.
More precisely, $0\leq \widehat{\mathcal{F}}(x,a,b) - \widehat{\mathcal{F}}(x,a',b) \leq a'-a$ for every $a\leq a'\leq b$.
	\item[(ii)] For every $x \in\Omega$ and $a\in\R$, the map $b\mapsto\widehat{\mathcal{F}}(x,a,b)$ is monotone non-decreasing in $[a,+\infty)$.
More precisely, $0\leq \widehat{\mathcal{F}}(x,a,b') - \widehat{\mathcal{F}}(x,a,b) \leq b'-b$ for every $a\leq b\leq b'$.
	\item[(iii)] For every $x \in\Omega$, the map $(a,b) \mapsto \widehat{\mathcal{F}}(x,a,b)$ is Lipschitz continuous in $D$.
More precisely, $|\widehat{\mathcal{F}}(x,a',b') - \widehat{\mathcal{F}}(x,a,b)| \leq |a'-a| + |b'-b|$
for every $(a,b), (a',b') \in D$.
\end{itemize}

For a fixed $u\in \BVL[\Omega]$, we consider the measure 
\begin{equation}\label{f:pairlsc}
	\Lpair{\bsmall(\cdot,u), Du}:= - \widehat{\mathcal{F}}(x,u^-(x),u^+(x))\, \sigma + \Div(\B(x,u(x))),
\end{equation}
where the subscript $L$ in the notation stands for \textit{Lower} semicontinuous. 
Since $F(x,\cdot)$ is a Lipschitz function for $\sigma$-a.e.\ $x\in\Omega$, 
for every $u\in\BV\cap L^\infty(\Omega)$ the set
\begin{equation*}
	\laM[u]  := \{\lambda\in\spacelambda\colon \widehat{\mathcal{F}}(x,\umeno(x),\upiu(x))= F(x,\Prec{u}(x))\ \text{for}\ \sigma\text{-a.e.}\ x\in \Omega\}
\end{equation*}
is not empty.
Specifically, let us consider the Borel set $A := \Omega\setminus(S_u\setminus J_u)$ and 
let us define the function
\[
g(x) := \max_{s\in [0,1]} F(x, (1-s) u^-(x) + s\, u^+(x)),
\qquad x\in A.
\]
Since $F$ is continuous with respect to the second variable, we have that
\[
g(x) = \sup_{q\in \mathbb{Q}\cap [0,1]} F(x, (1-q) u^-(x) + q\, u^+(x)),
\qquad x\in A,
\]
is a Borel-measurable function,
being the supremum of a countable family of Borel-measurable functions.
Hence, the function
\[
\psi(x,s) := g(x) - F(x, (1-s) u^-(x) + s\, u^+(x)),
\qquad
(x,s) \in A\times [0,1],
\]
satisfies all the assumptions of Lemma~\ref{l:meas},
so that there exists $\lambda\in\spacelambda[A]$,
that can be extended to a function $\widehat{\lambda}\in\spacelambda$, such that
$\psi(x, \widehat{\lambda}(x)) = 0$ for every $x\in A$, and hence for
$\sigma$-a.e.\ $x\in\Omega$.

By the definition of $\psi$, this is equivalent to say that
\begin{equation}
\label{f:uhat}
\widehat{\mathcal{F}}(x,\umeno(x),\upiu(x))= F(x,\uh(x)), 
\qquad
\uh := \Prec[\widehat\lambda]{u}\,,
\qquad \text{for $\sigma$-a.e.}\ x\in\Omega.
\end{equation}
In particular, $\uh$ is a Borel measurable function. Moreover,
\begin{equation*}%\label{f:lmpair}
\Lpair{\bsmall(\cdot,u), Du}= \pair{\bsmall(\cdot,u), Du}, \qquad \forall\lambda \in \laM[u],
\end{equation*}
so that Proposition \ref{p:tracesonJ} and Theorem \ref{chainb4} are in force.

We remark that, for every $\lambda\in\spacelambda$,
\begin{equation}\label{f:low}
\pscal{\Lpair{\bsmall(\cdot,u),Du}}{\varphi}
\leq
\pscal{\pair{\bsmall(\cdot,u), Du}}{\varphi},
\qquad
\forall \varphi\in C_c^1(\Omega),\ \varphi\geq 0.
\end{equation}

\begin{remark}
	We can give an explicit representation of $\laM[u]$  for separable fields, that is when
	there exists a decomposition $\Omega=\Omega^+ \cup \Omega^-$, with $\Omega^+\cap \Omega^- =\emptyset$,
	such that \eqref{splitomega1} holds.
	In this case,
	$F(x,\cdot)$ is a nondecreasing function on $\R$ for $\sigma$--a.e.\ $x\in \Omega^+$, and it is 
	a nonincreasing function on $\R$ for $\sigma$--a.e.\ $x\in \Omega^-$,
	so that
	\[
	\begin{split}
		\laM[u]  & = \{\lambda\in\spacelambda\colon \widehat{\mathcal{F}}(x,\umeno(x),\upiu(x))= F(x,\Prec{u}(x))\ \text{for $\sigma$-a.e. $x$\ in}\ \Omega \} \\
		&=
		\{\lambda\in\spacelambda\colon \lambda=0\ \text{$\sigma$-a.e.\ in }\ \Omega^+\cap J_u,\  \lambda=1\ \text{$\sigma$-a.e.\ in}\ \Omega^-\cap J_u\}.
	\end{split}
	\]
	In particular, in the autonomous case $\bsmall(x,t) \equiv \A(x)$ we recover the results proved in \cite{CDM}.
\end{remark}

We introduce the sequences along which we will compute our functionals. We add to the usual $L^1$-convergence a very mild version of convergence of the precise representatives of $BV$ functions,
and an equi-boundedness condition in $L^1(\Omega;\sigma)$.

\begin{definition}\label{d:conv}
Let $(u_n)_n$ be a sequence in $\BVL[\Omega]$
and let $u\in\BVL[\Omega]$.
We say that $(u_n)\in\Conv$ if the following properties hold.
\begin{itemize}
\item[(i)]  $u_n \to u$ in $L^1(\Omega)$.
\item[(ii)] The Pointwise Representative property (PR-property for short) holds, meaning that
\begin{equation}\label{f:lahti}
		\umeno(x) \leq \liminf_n u_{n}^-(x)\leq \limsup_n u_{n}^+(x)
		\leq \upiu(x)\,,\ \ \  {\text{$\sigma$-a.e. in}}\  \Omega.
	\end{equation}
\item[(iii)]
There exists $g\in L^1(\Omega; \sigma)$ such that
$|u_n^\pm| \leq g$, $\sigma$-a.e.\ in $\Omega$, for every $n\in\N$.
\end{itemize}
\end{definition}	

\begin{remark}\label{strictly}
We recall that by \cite[Theorem 3.2, and Corollary 3.3]{La},  
if a sequence $u_n$ strictly converges in $BV$ to $u$, 
then it admits a subsequence satisfying the PR-property $\mathcal{H}^{N-1}$-a.e.\ in $\Omega$, and hence $\sigma$-a.e. in $\Omega$.
\end{remark}

\begin{remark}\label{lebesgue}
If $\sigma \ll \mathcal{L}^N$, 
then any sequence $(u_n)$ converging to $u$ in $L^1(\Omega)$
admits a subsequence satisfying (ii) and (iii).
Hence, if $(u_n)$ and $u$ belong to $\BVL[\Omega]$ and $u_n\to u$ in $L^1(\Omega)$,
then there exists a subsequence $(u_{n_k})\in\Conv$.
\end{remark}

\begin{remark}
Condition (iii) is needed to apply dominated convergence in $L^1(\Omega; \sigma)$
when passing to the limit in the pairing
along sequences.
\end{remark}

The following essential tool, which guarantees the existence of approximating sequences with well-behaved pointwise values,  is proved in \cite[Theorem 3.2]{ComiLeo}.

\begin{theorem}[Comi--Leonardi $\lambda$-approximation]\label{l:CL}
Fixed $\lambda\in\spacelambda$, for every 
	$u \in\BVL[\Omega]$ there exists a sequence $(\ucl)_{n\in\N}\subset C^\infty(\Omega)\cap BV(\Omega)$ such that the following hold:
	\begin{enumerate}
		\item 
$\ucl \to u$ in $BV(\Omega)$-strict as $n \to +\infty$;

		\item 
$\ucl(x) \to u^\lambda(x)$ for $\H^{N-1}$--a.e.\ $x\in\Omega$ as $n \to +\infty$;

		\item 
$\|\ucl\|_{L^\infty(\Omega)}\leq \left( 1+\frac{1}{n}\right)\| u\|_{L^\infty(\Omega)}$.
	\end{enumerate}
	In particular, $(\ucl)\in \mathcal{A}(u)$.
\end{theorem}

\begin{lemma}\label{l:limCL}
Let $\bsmall$ satisfy assumptions $(b1)$-$(b4)$. 
Given $u\in \BVL[\Omega]$, and $\lambda\in\spacelambda$, let $(\ucl)_{n\in\N}$ be the sequence given by Theorem \ref{l:CL}. Then for every $\Lambda\in\spacelambda$
\begin{equation}\label{f:limCL}
\lim_{n \to \infty} \pscal{\pair[\Lambda]{\bsmall(\cdot,\ucl),D\ucl}}{\varphi}= \pscal{\pair{\bsmall(\cdot,u),Du}}{\varphi},
\quad \forall \varphi\in C^1_c(\Omega).
\end{equation}
\end{lemma}

\begin{proof}
	For every $\varphi\in C^1_c(\Omega)$ we have
	\[
	 \pscal{\pair[\Lambda]{\bsmall(\cdot,\ucl),D\ucl}}{\varphi}= \int_\Omega \B(x,\ucl) \nabla \varphi\, dx -
	 \int_\Omega F(x,\ucl)\varphi\, d\sigma,
	\]
	with right--hand side not depending on $\Lambda$ and with both terms that pass to the limit by (generalized) dominated convergence, due to \eqref{f:LipB}, \eqref{f:estiv}, and Theorem \ref{l:CL}.
\end{proof}

We first prove that, fixed $u\in \BVL[\Omega]$, then the $\lambda$-pairings for $\lambda\in \laM[u]$ are lower semicontinuous along the sequences in $\Conv$.

\begin{theorem}
	\label{t:lsc}
	Let $\bsmall$ satisfy assumptions $(b1)$-$(b4)$. 
Given $u\in \BVL[\Omega]$, let $\Lpair{\bsmall(\cdot,u),Du}$ be the pairing defined in \eqref{f:pairlsc}.
Then, for every $(u_n)\in\Conv$,
	it holds that
	\begin{equation}\label{f:lsc1}
		\pscal{\Lpair{\bsmall(\cdot,u),Du}}{\varphi}
		\leq
		\liminf_{n\to +\infty} \pscal{\Lpair{\bsmall(\cdot,u_n),Du_n)}}{\varphi}
		\qquad
		\forall \varphi\in C^\infty_c(\Omega),
		\
		\varphi\geq 0.
	\end{equation}
\end{theorem}

\begin{proof}
	Since $|\B(x, u_n(x)| \leq C |u_n(x)|$ (see \eqref{f:LipB}),
	by the (generalized) Dominated Convergence Theorem we have that
	\begin{equation}\label{f:limsup3}
		\lim_{n\to +\infty} \int_\Omega \B(x,u_{n}(x)) \cdot \nabla\varphi\, dx = \int_\Omega \B(x,u(x)) \cdot \nabla\varphi\, dx.
	\end{equation} 
	Let $x\in \Omega$ be such that \eqref{f:lahti} holds true. Then for every $\eps>0$ there exists $n_\eps\in\N$ such that
	\[
	u^-(x) -\eps \leq u_n^-(x) \leq  u_n^+(x) \leq u^+(x)+\eps, \qquad \quad \forall n\geq n_\eps.
	\]
	Using the monotonicity properties and the 1-lipschitzianity of $\widehat{\mathcal{F}}$, we have
	\[
	\widehat{\mathcal{F}}(x,u_{n}^-(x),u_{n}^+(x)) \leq \widehat{\mathcal{F}}(x,u^-(x)-\eps,u^+(x)+\eps) \leq 
	\widehat{\mathcal{F}}(x,u^-(x),u^+(x)) + 2\eps, \quad \forall n\geq n_\eps,
	\]
	which implies that
	\begin{equation}\label{f:liminf4}
	 \limsup_{n\to +\infty} \widehat{\mathcal{F}}(x,u_{n}^-(x),u_{n}^+(x)) \leq 
	 \widehat{\mathcal{F}}(x,u^-(x),u^+(x)), \qquad \text{for $\sigma$-a.e.}\ x \in \Omega.
	\end{equation}
Moreover, by Definition~\ref{d:conv}(iii), we also infer that
\[
\widehat{\mathcal{F}}(x,u_{n}^-(x),u_{n}^+(x)) \leq |u_n^+(x)| + |u_n^-(x)| \leq 2 g(x) \in L^1(\Omega;\sigma).
\]
and hence, by Fatou's Lemma and \eqref{f:liminf4}, we get
\begin{equation}\label{f:limsup1}
	\begin{split}
		& \lim_{n\to +\infty} \int_{\Omega}  \widehat{\mathcal{F}}(x,u_{n}^-(x),u_{n}^+(x))\, \varphi(x)\, d\sigma 
		\\
		& \leq
		\int_{\Omega} \left[ \limsup_{n\to +\infty} \widehat{\mathcal{F}}(x,u_{n}^-(x),u_{n}^+(x)) \right] \varphi(x)\, d\sigma  
		\leq 
		\int_{\Omega} \widehat{\mathcal{F}}(x,u^-(x),u^+(x))  \varphi(x)\, d\sigma.
	\end{split}
\end{equation}
Finally, 
from \eqref{f:limsup3} and \eqref{f:limsup1}
we conclude that
\[
\begin{split}
	\liminf_{n\to +\infty} & \pscal{\Lpair{\bsmall(\cdot,u_n),Du_n)}}{\varphi} \\ 
	& =
	-\limsup_{n\to +\infty}\left(
	\int_\Omega \B(x,u_{n}(x)) \cdot \nabla\varphi\, dx +
	\int_{\Omega}   \widehat{\mathcal{F}}(x,u_{n}^-(x),u_{n}^+(x))\, \varphi(x)\, d\sigma \right)
	\\ & 
	\geq
	-\int_\Omega \B(x,u(x)) \cdot \nabla\varphi\, dx -
	\int_{\Omega}   \widehat{\mathcal{F}}(x,u^-(x),u^+(x))\, \varphi(x)\, d\sigma
	\\ & =
	\pscal{\Lpair{\bsmall(\cdot,u),Du)}}{\varphi} 
\end{split}
\]
i.e., \eqref{f:lsc1} holds true.
\end{proof}

Taking into account Remark \ref{strictly} the following result holds.
\begin{corollary}
Let $\bsmall$
satisfy assumptions $(b1)$-$(b4)$.
Then,
for every $u_n, u\in \BVL[\Omega]$ such that $(u_n)$ strictly converges to $u$  in $BV$,
and satisfies property (iii) in Definition~\ref{d:conv},
the lower semicontinuity property \eqref{f:lsc1} holds.
\end{corollary}

We now prove the converse of Theorem~\ref{t:lsc}.

\begin{theorem}
	\label{t:lscvice}
	Let $\bsmall$
	satisfy assumptions $(b1)$-$(b4)$.
	Let $\lambda\in\spacelambda$, and $u\in \BVL[\Omega]$. 
Assume that, 
for every sequence $(u_n)\in\Conv$,
	it holds that
	\begin{equation}\label{f:lsc1333}
		\pscal{\pair{\bsmall(\cdot,u),Du)}}{\varphi}
		\leq
		\liminf_{n\to +\infty} \pscal{\pair[{\lambda}]{\bsmall(\cdot,u_n),Du_n}}{\varphi}
		\qquad
		\forall \varphi\in C^\infty_c(\Omega),
		\
		\varphi\geq 0.
	\end{equation}
Then we have that $\lambda\in \laM[u]$, i.e.
$\pair{\bsmall(\cdot,u), Du}=\Lpair{\bsmall(\cdot,u), Du}$.
\end{theorem}
\begin{proof}
Let $\lambda$ satisfy \eqref{f:lsc1333}. 
Let $\uh$ be the function defined in~\eqref{f:uhat},
and let $u_n\in C^\infty(\Omega)\cap\Conv[\uh]$ be the sequence considered in Theorem~\ref{l:CL}.
By Lemma~\ref{l:limCL} we have that 
\begin{equation*}
\lim_{n \to \infty} \pscal{\pair{\bsmall(\cdot,\ucl),D\ucl}}{\varphi}= \pscal{\Lpair{\bsmall(\cdot,u),Du}}{\varphi},
\quad \forall \varphi\in C^1_c(\Omega).
\end{equation*}
Hence \eqref{f:lsc1333} gives
\[
\pair{\bsmall(\cdot,u), Du}\leq \Lpair{\bsmall(\cdot,u), Du},
\]
and the conclusion now follows recalling~\eqref{f:low}. 
\end{proof}

\begin{remark}
In the autonomous case~\eqref{autonomous},
considered in Remark~\ref{r:auto}, it is known that, in fact,
for every $u_n, u\in \BVL[\Omega]$ such that 
$(u_n)$ converges to $u$ in $L^1$
	it holds that
	\begin{equation*}%\label{f:lsc1333000}
		\pscal{(\bsmall(u),Du))}{\varphi}
		=
		\lim_{n\to +\infty} \pscal{(\bsmall(u_n),Du_n)}{\varphi}
		\qquad
		\forall \varphi\in C^1_c(\Omega),
		\
		\varphi\geq 0
	\end{equation*}
	and 	
\begin{equation*}%\label{f:lscdd}
		\int_\Omega|(\bsmall(u),Du)|
		\leq
		\liminf_{n\to +\infty} \int_\Omega|(\bsmall(u_n),Du_n)|\,,
\end{equation*}
	(see \cite[Lemma 9]{DeCic}).
\end{remark}

In the following proposition we give an explicit representation of the jump part
of the lower-semicontinuous pairing.

\begin{proposition}
	\label{p:tracesonJ22}
	Let $\bsmall$ satisfy assumptions $(b1)$-$(b4)$, 
let $u\in\BV\cap L^\infty(\Omega)$ and assume $J_u$ oriented in such a way that $u^i=u^+$ and $u^e=u^-$. Then 
\begin{equation*}
	%\label{f:tronJ12}
	\Lpair{\bsmall(\cdot,u), Du}  \res J_u
	=\Trl[L]{\bsmall}{J_u}\,  |D^ju|,
\end{equation*}
where
\[
\begin{split}
\Trl[L]{\bsmall}{J_u}(x) :=
\min_{s\in[u^-(x),u^+(x)]} %\Big\{ 
\mean{u^-(x)}^{u^+(x)} &
\Big[\chi_{[u^-(x), s]}(t)\,\Trm{\bsmall_t}{J_u}(x)
\\ & 
+\chi_{[s, u^+(x)]}(t)\,\Trp{\bsmall_t}{J_u}(x)\Big]\, dt
.
\end{split}
\]
\end{proposition}

\begin{proof}
The result follows from a straightforward computation:
	\[
		\begin{split}
			\Lpair{\bsmall(\cdot,u), Du}  \res J_u  &=
			\Big[\beta^i(x,u^+(x))- \beta^e(x,u^-(x))-\max_{s\in[u^-,u^+]}[\beta^i(x,s)-\beta^e(x,s)] \Big]\H^{N-1} \res J_u
			\\  &=\min_{s\in[u^-,u^+]}\Big[\beta^i(x,u^+(x))- \beta^e(x,u^-(x))-\beta^i(x,s)+\beta^e(x,s) \Big]\H^{N-1} \res J_u
			\\  &=\min_{s\in[u^-,u^+]}\left[\int_{u^-}^{s}\gamma^e(x,t)\,dt+\int_{s}^{u^+}\gamma^i(x,t)\,dt\right]
			\H^{N-1} \res J_u
			\\  &=\min_{s\in[u^-,u^+]}\left[\int_{u^-}^{u^+}[\chi_{[u^-, s]}(t)\,\gamma^e(x,t)
+\chi_{[s, u^+]}(t)\,\gamma^i(x,t)]\,dt\right] \H^{N-1} \res J_u
			\\  
& =\min_{s\in[u^-,u^+]}\left[\mean{u^-}^{u^+}[\chi_{[u^-, s]}(t)\,\gamma^e(x,t)+\chi_{[s, u^+]}(t)\,\gamma^i(x,t)]\,dt\right] |D^ju|.
		\end{split}
	\]
\end{proof}

The total variation of the pairing measure 
$\Lpair{\bsmall(\cdot,u),Du}$ is not, in general,
lower semicontinuous along sequences
$(u_n)\in\Conv$.
Therefore, we proceed to characterize the selections
$\lambda\in\spacelambda$ that guarantee it.
More precisely, 
given $u\in\BVL[\Omega]$,
for every $x\in A:= \Omega\setminus(S_u\setminus J_u)$
let us define the function
\begin{equation}
\label{f:zeta}
\zeta(x,t) :=
\int_{u^-(x)}^{t}\gamma^e(x,\ell)\,d\ell+\int_{t}^{u^+(x)}\gamma^i(x,\ell)\,d\ell,
\qquad
t\in[u^-(x),u^+(x)]
\end{equation}
(see, for comparison, Proposition~\ref{p:tracesonJ22}),
and let
\[
g(x) := \min_{s\in [0,1]} |\zeta(x, (1-s) u^-(x) + s\, u^+(x))|\,,
\qquad x\in A.
\]
The function $g$ is Borel-measurable, since it can be written as the infimum
of a countable family of Borel-measurable functions:
\[
g(x) = \inf_{q\in [0,1]} |\zeta(x, (1-q) u^-(x) + q\, u^+(x)),
\qquad x\in A.
\]
Hence, the function
\[
\psi(x,s) :=  |\zeta(x, (1-s) u^-(x) + s\, u^+(x))| - g(x),
\qquad
(x,s) \in A\times [0,1],
\]
satisfies the assumptions of Lemma~\ref{l:meas},
so that there exists a function $\lambda\in\spacelambda[A]$,
that can be extended to a function $\lambda^V\in\spacelambda$,
such that $\psi(x, \lambda^V(x)) = 0$ for every $x\in A$.

By construction, the function $u^V := u^{\lambda^V}$ satisfies
\[
|\zeta(x, \Vu(x))| = \min_{t\in [u^-(x), u^+(x)]} |\zeta(x, t)|\,,
\qquad
\text{for $\sigma$-a.e.}\ x\in\Omega.
\]
The class of selections $\lambda\in\spacelambda$ 
for which $|\Vpair{\bsmall(\cdot,u), Du}|$ is lower semicontinuous along sequences
$(u_n)\in\Conv$ is 
\begin{equation*}
	\laV[u]  := \Big\{\lambda\in\spacelambda\colon 
	\lambda(x) = \dfrac{u^V(x) - u^-(x)}{u^+(x)-u^-(x)}
\ \text{for $\sigma$-a.e.}\ x\in J_u\Big\}.
\end{equation*}
This class is not empty since, by construction, $\lambda^V\in \laV[u]$.
We can then define the following pairing measure,
\begin{equation}
\label{f:vpair}
\Vpair{\bsmall(\cdot,u), Du}= \pair{\bsmall(\cdot,u), Du}, \qquad \forall\lambda \in \laV[u],
\end{equation}
which is clearly independent of the choice of
$\lambda\in\laV[u]$.
Moreover, it is straightforward to check that,
\[
|\Vpair{\bsmall(\cdot,u),Du}|
\leq |\pair{\bsmall(\cdot,u),Du}|
\qquad\forall \lambda\in\spacelambda.
\]

The following result,
analogous to Theorem~\ref{t:lsc}, holds.

\begin{theorem}
	\label{t:lscv}
	Let $\bsmall$ satisfy assumptions $(b1)$-$(b4)$. 
Fixed $u\in \BVL[\Omega]$, let $\Vpair{\bsmall(\cdot,u),Du}$ be the pairing defined in \eqref{f:vpair}.
Then, for every $(u_n)\in\Conv$,
it holds that
\begin{equation}\label{f:lsc2}
	\int_\Omega|\Vpair{\bsmall(\cdot,u),Du}|
	\leq
	\liminf_{n\to +\infty} \int_\Omega|\Vpair{\bsmall(\cdot,u_n),Du_n)}|\,.
\end{equation}
\end{theorem}

\begin{proof} 
We can partition
$J_u$ into three disjoint Borel sets $J_u^m$, $J_u^M$, and $J_u^0$
such that
\[
\begin{cases}
m_x := \min_{t\in [u^-(x), u^+(x)]} \zeta(x,t) > 0
& \forall x\in J_u^m,
\\
M_x := \max_{t\in [u^-(x), u^+(x)]} \zeta(x,t) < 0 
&\forall x\in J_u^M,
\\
m_x \leq 0 \leq M_x
& \forall x\in J_u^0,
\end{cases}
\]
where $\zeta$ is the function defined in ~\eqref{f:zeta}.
Given $\lambda\in\laV[u]$, the density $\Theta_V$
of $|\Vpair{\bsmall(\cdot,u),Du}|$ with respect to $|Du|$
satisfies
\begin{equation}
\label{f:ThetaV}
{\Theta_V}(x) = |\zeta(x, \Vu(x))| =
\begin{cases}
m_x, &\text{for $\sigma$-a.e.\ $x\in J_u^m$},
\\
-M_x, &\text{for $\sigma$-a.e.\ $x\in J_u^M$},
\\
0, &\text{for $\sigma$-a.e.\ $x\in J_u^0$}.
\end{cases}
\end{equation}

Hence, if we denote by $\Theta_G$
and $\Theta_H$ the densities of the measures
$\left((\bsmall(x, u),Du)_L\right)^+$ 
and
$\left((-\bsmall(x, u),Du)_L\right)^+$ respectively,
then it holds that
\begin{gather*}
	\Theta_G = \Theta_V,
	\quad \text{$\sigma$-a.e.\ on $J^m_u$},
	\quad
	\Theta_G = 0
	\quad \text{$\sigma$-a.e.\ on $J^M_u \cup J^0_u$},
	\\
	\Theta_H = \Theta_V,
	\quad \text{$\sigma$-a.e.\ on $J_u^M$},
	\quad
	\Theta_H = 0
	\quad \text{$\sigma$-a.e.\ on $J^m_u \cup J^0_u$}.
\end{gather*}
whereas $\Theta_V = \Theta_G + \Theta_H$
$\sigma$-a.e.\ on $\Omega\setminus J_u$.

Then 
\[
F(u):=\int_\Omega |(\bsmall(x, u),Du)_V|=G^+(u)+H^+(u),
\]
where
$G^+, H^+\colon \BVL[\Omega]\to [0,+\infty[$ are defined by
\[
\quad
G^+(u):= \int_\Omega \left((\bsmall(x, u),Du)_L\right)^+\,
\qquad 
\qquad
H^+(u):= \int_\Omega\left((-\bsmall(x, u),Du)_L\right)^+\,.
\]

The conclusion then follows once one proves that the functionals $G^+(u)$ and $H^+(u)$ are lower semicontinuous along sequences $(u_n)\in\Conv$. 
Since
$$
G^+(u)=\sup_{\{\varphi\in C^\infty_c(\Omega),
		0\leq\varphi\leq 1\}}\pscal{\Lpair{\bsmall(\cdot,u),Du}}{\varphi},
$$
this is a consequence of Theorem \ref{t:lsc}.
The proof for $H^+$ is similar.
\end{proof}

\section{Pairing as relaxed functional}
\label{s:relax}

For every function $\varphi\in C^1_c(\Omega)$ 
let us consider  the functional 
defined by
\[
\mathcal{F}^\varphi(u):=
\begin{cases}
\displaystyle\int_\Omega \varphi\,\bsmall(x,u)\cdot\nabla u\,dx
& \text{$u\in W^{1,1}(\Omega)$}\cap L^\infty(\Omega),
\\
+\infty
,
& u\in (BV(\Omega)\cap L^\infty(\Omega))
\setminus W^{1,1}(\Omega).
\end{cases}
\]
Moreover, for every $u\in\BVL[\Omega]$
we define the relaxation
\[
\overline{\mathcal{F}^\varphi}(u) :=
\inf
\left\{
\liminf_{n\to +\infty}
\mathcal{F}^\varphi(u_n)\colon
(u_n) \in  W^{1,1}(\Omega)\cap\Conv 
\right\}
\,,
\]
where $\Conv$ is the class introduced in Definition~\ref{d:conv}.

The following result was proved in \cite[Theorem 7.1]{CD5}
under the additional assumption $\Div_x \bsmall_t \in L^1(\Omega)$.
Here we provide a full generalization to the case of
divergence-measure vector fields.

\begin{theorem}
\label{t:limsupG}
Let $\bsmall$
satisfy assumptions $(b1)$-$(b4)$.
Then, 
for every $\varphi\in C^1_c(\Omega)$, it holds that
\begin{equation}\label{f:conv}
\overline{\mathcal{F}^\varphi}(u)=\int_\Omega\varphi\, d (\bsmall(x,u),Du)_L\,,
\qquad
u\in\BVL[\Omega]\,.
\end{equation}
\end{theorem}

\begin{proof}
Firstly, we notice that Theorem \ref{t:lsc} guarantees the inequality
\begin{equation}\label{f:lb}
\overline{\mathcal{F}^\varphi}(u)\geq\int_\Omega\varphi\, d (\bsmall(x,u),Du)_L.
\end{equation}
On the other hand, 
for every $u\in  \BVL[\Omega]$, we obtain the recovery sequence as follows. 
For $\lambda\in \laM[u]$ be such that  $\widehat{\mathcal{F}}(x,\umeno(x),\upiu(x))= F(x,\Prec{u}(x))$, let  $(\ucl)_{n \in\N}\subset C^\infty(\Omega)\cap \mathcal{A}(\Prec{u})$ be as in Theorem~\ref{l:CL}.
By Lemma \ref{l:limCL} we have that
\begin{equation}\label{f:ls4}
	\lim_{n\to +\infty} \mathcal{F}^\varphi(\ucl)= \overline{\mathcal{F}^\varphi}(u)\,,
\end{equation}
which, combined with \eqref{f:lb}, leads to \eqref{f:conv}.
\begin{equation}\label{f:ls1}
	\begin{split}
		\mathcal{F}^\varphi(\ucl)=&\int_\Omega \varphi\,\bsmall(x,\ucl)\cdot\nabla \ucl\,dx\\
		=&
		-\int_\Omega F(x,\ucl)\varphi\,d\sigma
		-\int_\Omega \B(x, \ucl)\cdot \nabla\varphi\,dx.
	\end{split}
\end{equation}
which, combined with \eqref{f:lb}, leads to \eqref{f:conv}.
\end{proof}

The above relaxation result is clearly related to the semicontinuity results
proved in Theorems~\ref{t:lsc} and~\ref{t:lscvice}.

\medskip
On the other hand,
in view of the semicontinuity result stated in Theorem~\ref{t:lscv},
we conjecture that an analogous relaxation result holds true also for
$|\Vpair{\bsmall(\cdot,u), Du}|$,
where $\Vpair{\bsmall(\cdot,u), Du}$ is the pairing measure introduced in~\eqref{f:vpair}. 

Indeed, this relaxation result, with respect to the $L^1$ convergence,
has been already proved in the linear case
$\B(x,t) = \A(x) \, t$, 
with $\A \in BV(\Omega, \R^N)$.
Specifically, in this setting the function $f(x,\xi) := |\A(x)\cdot\xi|$
satisfies %assumption~(1.5) of Theorem~1.1 in~\cite{ADCF},
the assumptions of Theorem~4.2 in~\cite{ADCF},
where it is proved that the relaxation of the functional
\[
\mathcal{F}(u):=
\begin{cases}
\displaystyle\int_\Omega |\A(x)\cdot\nabla u|\,dx
& \text{$u\in W^{1,1}(\Omega)$}\cap L^\infty(\Omega),
\\
+\infty
,
& u\in (BV(\Omega)\cap L^\infty(\Omega))
\setminus W^{1,1}(\Omega),
\end{cases}
\]
with respect to the $L^1$ convergence,
is in fact
\[
\overline{\mathcal{F}}(u) =
|\Vpair{\A, Du}| (\Omega),
\qquad \forall u\in \BVL[\Omega].
\]
In \cite[formula~(1.7)]{ADCF} the relaxed functional $\overline{\mathcal{F}}$ is expressed
in terms of the function
\[
f^-(x,\xi) = \min\{|\A^+(x)\cdot\xi|,\, |\A^-(x)\cdot\xi|\}.
\]
From the explicit representation given in~\eqref{f:ThetaV}, we have that
$f^-(x, \nu_u)$ coincides with the density $\Theta_V$
of $|\Vpair{\A,Du}|$ with respect to $|Du|$.

\smallskip
The main obstruction in proving the relaxation result 
in the general case is to pass to the limit along the recovering sequence,
as in \eqref{f:ls1}. 
Indeed, due to the presence of the absolute value, we cannot integrate by parts
the term containing $\Div (\B(x, u(x))$
and entirely different techniques should be used.

\bigskip\bigskip
\noindent
{\bf Acknowledgments}. 
The authors are members of  the Istituto Nazionale di Alta Matematica (INdAM), GNAMPA Gruppo Nazionale per l'Analisi Matematica, la Probabilità e le loro Applicazioni, and are partially supported by the INdAM--GNAMPA 2024 Project \textit{Pairing e div-curl lemma: estensioni a campi debolmente derivabili e diﬀerenziazione non locale}.

\def\cprime{$'$}

\end{document}